\documentclass[10pt, a4paper, reqno]{amsart}
\usepackage{macros-amsart-technical}
\usepackage{macros-amsart-maths}
\usepackage{amsaddr}

\title[Robust Pontryagin Maximum Principle]{Robust Discrete-Time Pontryagin Maximum Principle on Matrix Lie Groups}
\author[A. A. Joshi, D. Chatterjee and R. N. Banavar]{Anant A. Joshi, Debasish Chatterjee and Ravi N. Banavar}
\address{Department of Systems and Control Engineering,\\ IIT Bombay, Powai\\ Mumbai 400076, India\\ {\tt  \{anantjoshi,dchatter,banavar\}@iitb.ac.in}}

\date{\today}

\begin{document}
\maketitle

\begin{abstract}

This article considers a discrete-time robust optimal control problem on matrix Lie groups. The underlying system is assumed to be perturbed by exogenous unmeasured bounded disturbances, and the control problem is posed as a min-max optimal control wherein the disturbance is the adversary and tries to maximise a cost that the control tries to minimise.
	Assuming the existence of a saddle point in the problem, we present a version of the Pontryagin maximum principle (PMP) that encapsulates first-order necessary conditions that the optimal control and disturbance trajectories must satisfy. This PMP features a saddle point condition on the Hamiltonian and a set of backward difference equations for the adjoint dynamics. We also present a special case of our result on Euclidean spaces.
	We conclude with applying the PMP to robust version of single axis rotation of a rigid body.

\end{abstract}

\section{Introduction}
Optimal control is a cornerstone in the theory of modern control \cite{agrachev,liberzon,sarg-2000}.
The Pontryagin maximum principle \cite{pontryagin} (henceforth referred to as the PMP) constitutes an integral part of optimal control theory, and it provides first-order necessary conditions for optimality.
These necessary conditions make the process of finding an optimal control significantly simpler by narrowing the set of admissible controls into a subset of candidate optimisers, and they are used by numerical algorithms to search for an optimal control. It is, in fact, one of the most widely used tools in optimal control apart from the Hamilton-Jacobi-Bellman (HJB) approach. 
The PMP has traditionally been studied in a continuous-time setting, and while the discrete-time version \cite{bolt} received less attention in the initial stages, there have recently been several investigations into the discrete time PMP, for instance, in the presence of constraints and in the geometric setting; see, e.g., \cite{gupta-2019, mishal-2020,prad-2019,karmvir-2018} and the references therein.

The mathematical model of any engineering system is almost always an approximation of its actual behaviour, and consequently, it is essential to address the robustness aspect of control schemes for practical applications. By this we mean the ability of the controller to give satisfactory performance under system modelling uncertainties and/or the presence of exogenous unmeasured disturbances.  Our emphasis here will be on deterministic controllers formulated as min-max optimisation problems (also studied as games), in which the disturbance is the adversary and tries to maximise the cost which the control tries to minimise  \cite{basar-1995,basar-1998,chen-1997, chern-2005,raimondo-2009}.

PMPs for robust optimal control problems of the aforementioned type have been derived in \cite{vinter-2005} for continuous time systems with uncertainties, and in \cite{bolt-robust} for systems that have parametric uncertainties. For continuous time systems, there has been considerable work in the game theoretic framework on the min-max control for systems with disturbances, we pick the representative articles \cite{bernhard-2015,bressan-2011} from that literature since they contain results that closely resemble the PMP. When \cite{bressan-2011} is specialized to our setting of a min-max problem, one arrives at a zero sum differential game which has a Nash equilibrium (a game theoretic analog of a saddle point). This gives rise to two distinct sub-problems, one of minimising the cost over admissible controls (and this problem is parametrised by the optimal disturbance), and the other of maximising the cost over admissible disturbances (which is paramerised by the optimal control). The author proceeds to apply the PMP on both sub-problems individually to arrive at two sets of two-point boundary value problems, each with their own Hamiltonians and covectors, each parametrized by the solution of the other. \cite{bernhard-2015} takes this further by deriving a single maximum principle with a single Hamiltonian and single covector by employing the Isaacs equation \cite{galperin-2008} in an interesting fashion. Our approach here is derived from the ideas in these two articles, but we operate under a different regime that we now explain.
 
A broad class of aerospace and mechanical systems evolve on the class of smooth manifolds known as Lie groups. These state spaces lack a vector space structure, which warrants new techniques to be developed in order to study these systems, and such techniques fall under the broad umbrella of geometric mechanics and control \cite{agrachev,marsden}. 
Optimal control of such systems has received significant attention \cite{agrachev,bloch-2015}; in particular, PMPs for systems evolving on smooth manifolds has been studied both in continuous and discrete time \cite{gupta-2019,mishal-2020,karmvir-2018} in considerable detail. 
More specifically, geometric discrete-time optimal control, designed specifically to respect the non-flat nature of the underlying state spaces, has been of great interest in the engineering community since the associated techniques eliminate problems that occur by using local parametrisation \cite{bloch-2006,fernando-2013,karmvir-2018-jgcd}.

In \cite{karmvir-2018}, a discrete-time PMP for optimal control problems for systems evolving on matrix Lie groups was presented. That work however did not consider the performance of the controller under the effect of exogenous unmeasured disturbances. The current work takes the problem a step further by incorporating the effect of bounded disturbances acting on the class of systems considered in \cite{karmvir-2018} and by posing the optimal control problem as a min-max problem. Assuming the existence of a saddle point of the cost function, we provide first-order necessary conditions that the control and disturbance satisfy for optimality are presented in the form of a modification of the PMP. We arrive at a saddle point condition on the Hamiltonian and backward difference equations for the covector (also termed adjoint) dynamics. We also present a specialised version of our result to Euclidean spaces for those interested in directly applying it to such systems.
Our results are similar in spirit to \cite{bernhard-2015} but different in three very significant aspects:
\begin{itemize}[label=\(\circ\), leftmargin=*, align=left]
	\item We consider a problem that evolves in discrete-time, for which it is not possible to derive the maximum principle from a min-max version of the Bellman's equation. This is relevant, since in continuous-time the PMP can be derived using ideas from the HJB partial differential equation \cite[Chapter 5]{liberzon}, and similarly the modification of PMP in \cite{bernhard-2015} can be derived using ideas from the Isaacs equation.
	\item \cite{bernhard-2015} hypothesizes that the Hamiltonian satisfies a saddle point condition \emph{before} establishing their result. We do not do make any such assumption, but instead the saddle point condition appears naturally in our development. 
	\item \cite{bernhard-2015} assumes that the abnormal multiplier used in the Hamiltonian is non-zero at the outset, but we \emph{prove} that it must always be non-zero in our setting.
\end{itemize}
The results closest in spirit to ours are in \cite{bolt-robust}, but as noted earlier, \cite{bolt-robust} considers parametric uncertainty in the system.

This paper is organized as follows. 
Section \ref{sec:prelim} contains preliminaries and the statement of the maximum principle for our case, which is proven in Section \ref{sec:proof}.
Section \ref{sec:example} contains an aerospace example where our theory is applied.

\section{Preliminaries and Statement of Main Result}
\label{sec:prelim}
\subsection{Mathematical Preliminaries}

We present some mathematical preliminaries in this subsection. 
$\mathbb{W}$ will denote the set of non-negative integers.

Given a vector space $\V$, let $\V^*$ denote its dual space, which is the set of all linear functionals (also termed covectors) on $\V$. 
Denote by $\inprod{\cdot}{\cdot}: \V^* \times \V \to \R$ the duality pairing. 
Given a linear map between two vector spaces $\map: \V_{1} \to \V_{2}$, let $\map^*: (\V_{2})^* \to (\V_{1})^*$ denote the dual of $\map$ defined as $\inprod{\map^*(\eta)}{v} = \inprod{\eta}{\map(v)}$ $\forall \eta \in (\V_{2})^*, v \in \V_{1}$. The standard inner product on $\R^n$ will also be denoted by
$\inprod{\cdot}{\cdot}$ since it is equivalent to the duality pairing on $\R^n$. 
For a smooth map between two vector spaces $\map: \V_1 \to \V_2$, for any $x \in \V_1$, $\derivo{\map}{x}$ will denote the derivative of $\map$ at $x$ and $\derivosecond{\map}{x}$ will denote the second derivative of $\map$ at $x$. 
Given a third vector space $\V_3$, for a smooth map $\V_1 \times \V_2 \ni (x_1,x_2) \mapsto \map(x_1,x_2) \in \V_3$, for any $(\bar{x}_1,\bar{x}_2) \in \V_1 \times \V_2$,  $\deriv{x_1}{\map}{\bar{x}_1,\bar{x}_2}$ denotes the derivative of $\map(\cdot,\bar{x}_2)$, evaluated at $\bar{x}_1$ and $\derivsecond{x_1}{\map}{\bar{x}_1,\bar{x}_2}$ denotes the second derivative of $\map(\cdot,\bar{x}_2)$, evaluated at $\bar{x}_1$. 
If two vector spaces $\V_1$ and $\V_2$ are isomorphic, it will be denoted by $\V_1 \cong \V_2$. 
The preceding material was from  \cite{halmos,spivak}.

Given a cone $\cone \subset \R^n$ with vertex at $x \in \R^n$, we define the dual cone of $\cone$ as 
$
    \dual{\cone}(x) \defn \set{a \in \R^n}{\inprod{a}{x'-x} \geq 0 \forall x' \in K}.
$
For a convex set $\Omega \subset \R^n$, its supporting cone with vertex at $x \in \Omega$ is defined as 
$
\cone_{\Omega}(x) \defn \cl({\bigcup}_{{\alpha>0}} \set{x + \alpha (x' - x )}{x' \in \Omega}
)$.
A family of convex cones is defined to be \textit{separable} if there exists a hyperplane that separates one of them from the intersection of the others. If the family is not \textit{separable} then it is defined to be \textit{inseparable}. An interested reader is referred to \cite[Chapter~6]{bolt-robust} for a more elaborate treatment.

Consider a smooth function $\map: \R^n \to \R$. Suppose that we desire to find $\argmin \map(x)$ with $x \in \Sigma \subset \R^n$. Then, as per \cite[Theorem~6.1]{bolt-robust} $\opt{x} \in \Sigma$ is a minimum of $\map$ over $\Sigma$ if and only if $\Sigma \cap \Omega = \{\opt{x}\}$, where $\Omega \defn \set{x \in \R^n}{\map(x) < \map(\opt{x}) } \cup \{\opt{x}\}$. 

\begin{definition}[{\cite[Definition 11.4]{gueler}}]
	\label{defn:SP}
	Consider two arbitrary sets $\Uc$ and $\D$ and an arbitrary function $\map : \Uc \cross \D \to \R$. $(\opt{\con},\opt{\dist}) \in \Uc \times \D$ is a saddle point of $\map$ if
	\[
		\map(\opt{\con},\dist) \le \map(\opt{\con},\opt{\dist}) \le \map(\con,\opt{\dist}) \forall \con \in \Uc, \dist \in \D
	\]
\end{definition}

\begin{proposition}
\label{prop:saddle-point}
Consider a smooth function $\map : \R^{\dimu} \times \R^{\dimd} \to \R$ and $\Uc \times \D \subset \R^{\dimu} \times \R^{\dimd}$ with $(\opt{\con},\opt{\dist}) \in \Uc \times \D$. Define $\Omega_1 \defn \set{(u,\opt{d}) \in \R^{\dimu} \times \R^{\dimd}}{\map(u,\opt{d}) < \map(\opt{u},\opt{d})}$, $\Omega_2 \defn \set{(\opt{u},{d}) \in \R^{\dimu} \times \R^{\dimd}}{\map(\opt{u},{d}) > \map(\opt{u},\opt{d})}$, $\Omega_1^\prime \defn \Omega_1 \cup  \{ (\opt{\con},\opt{\dist}) \} $ and $\Omega_2^\prime \defn \Omega_2 \cup  \{ (\opt{\con},\opt{\dist}) \} $. The following are equivalent
\enspl{P}{
\item $(\opt{\con},\opt{\dist})$ is a saddle point of $\map$ restricted to $\Uc \times \D$. \label{list:saddle-i}
\item $\big(\Omega_1 \cup \Omega_2 \cup \{ (\opt{\con},\opt{\dist}) \} \big) \cap (\Uc \times \D ) = \{ (\opt{\con},\opt{\dist}) \}$ \label{list:saddle-ii}
\item $\Omega_1^\prime \cap (\Uc \times \D )  = \{ (\opt{\con},\opt{\dist}) \}$ and $\Omega_2^\prime \cap (\Uc \times \D )  = \{ (\opt{\con},\opt{\dist}) \}$ \label{list:saddle-iii}
}
\end{proposition}
A proof of Proposition \ref{prop:saddle-point} is provided Appendix \ref{app:proof-saddle-prop}.

Since the basic aim of the content of this paper is to find necessary conditions for a saddle point, Proposition \ref{prop:saddle-point}  yields a procedure for it as follows: 
\enspl{SP}{
\item Freeze $\dist$ at $\opt{\dist}$ and obtain necessary conditions for $\opt{\con}$ to be a minimum of $\map(\cdot,\opt{\dist})$ over $\Uc \cross \{\opt{\dist}\}$. 
These conditions will naturally be parametrised by $\opt{\dist}$. \label{list:SP1}
\item Freeze $\con$ at $\opt{\con}$ and obtain necessary conditions for $\opt{\dist}$ to be a maximum of $\map(\opt{\con},\cdot)$ over $\{\opt{\con}\} \cross \D$.
These conditions will naturally be parametrised by $\opt{\con}$. \label{list:SP2}
\item Both sets of necessary conditions should be simultaneously satisfied, however, notice that the multipliers that appear  in both sets of necessary conditions may be different. \label{list:SP3}
}
We will use the next result in our numerical simulations.

\begin{proposition}[Sufficient condition for saddle point]
Consider a smooth function $ \R^{\dimu} \times \R^{\dimd} \ni (\con,\dist) \mapsto \map(\con,\dist) \in R$ 
and a point $(\opt{\con},\opt{\dist}) \in \R^{\dimu} \times \R^{\dimd} $. Suppose that 
\en{
\item $\deriv{\con}{\map}{\opt{\con},\opt{\dist}} = 0$ and  $\derivsecond{\con}{\map}{\opt{\con},\opt{\dist}}$ is positive definite \label{list:suff-i}
\item $\deriv{\dist}{\map}{\opt{\con},\opt{\dist}} = 0$ and  $\derivsecond{\dist}{\map}{\opt{\con},\opt{\dist}}$ is negative definite \label{list:suff-ii}
}
Then there exist open sets $\Uc \subset \R^{\dimu}$ and $\D \subset \R^{\dimd}$ with $\opt{\con} \in \Uc$ and $\opt{\dist} \in \D$ such that $(\opt{\con},\opt{\dist})$ is a saddle point of $\map$ restricted to $\Uc \times \D$.
\label{prop:suff-saddle-point}
\end{proposition}

\begin{proof}
Condition \ref{list:suff-i} is a sufficient condition for $\opt{\con}$ to be a strict local minimum of $\map(\cdot,\opt{\dist})$, yielding that there exists an open set $\Uc \subset \R^{\dimu}$ with $\opt{\con} \in \Uc$ such that $\map(\opt{\con},\opt{\dist}) < \map(\con,\opt{\dist}) \forall \con \in \Uc$ \cite[Theorem 2.13]{gueler}. This gives the second inequality in Definition \ref{defn:SP}. A parallel argument for $F(\opt{\con},\cdot)$ gives the first inequality in Definition \ref{defn:SP}.
\end{proof}
%
We present some background on smooth manifolds. For more on these topics, we refer the reader to standard texts on differential geometry, for instance,  \cite[Chapter 3, 11]{lee}. For a smooth manifold $\M$, $\tspace{x}{\M}$ represents its tangent space at $x \in \M$ \cite[Page 54]{lee}.
Given two smooth manifolds $\M_{1}$ and $\M_{2}$, and a smooth map $\map: \M_{1} \to \M_{2}$, $\tango{x}{\map}: \tspace{x}{\M_{1}} \to \tspace{\map(x)}{\M_{2}}$ will denote the tangent map of $\map$  at $x$ (recall that the tangent map is the generalisation of the derivative in the context of smooth manifolds \cite[Page 68]{lee}). The cotangent map of $\map$ at $x$ (which is the dual of the tangent map \cite[Page 284]{lee}) will be denoted as $\tangostar{x}{\map}:  \left(\tspace{\map(x)}{\M_{2}}\right)^* \to \left(\tspace{x}{\M_{1}}\right)^*$, and for any $\eta \in \left(\tspace{\map(x)}{\M_{2}}\right)^*$, $\tangostar{x}{\map}(\eta)$ will denote its action on $\eta$.
Now, given a third smooth manifold $\M_3$, let $ \M_{1} \times \M_{2} \ni (x_{1},x_{2}) \mapsto \map(x_{1},x_{2}) \in \M_{3}$. Then 
$\tang{(\bar{x}_{1},\bar{x}_{2})}{x_{1}}{\map}$ 
denotes the tangent map of $\map(\cdot,\bar{x}_{2})$, evaluated at $\bar{x}_{1}$. 
The following notion will be used while defining covectors as tangent maps of real-valued functions. Given a smooth map $\map_{1} : \M_{1} \to \R$, for any $x \in \M_{1}$, $\tango{x}{\map_{1}} \in \left(\tspace{x}{\M_{1}}\right)^*$ since $\tango{x}{\map_{1}} \cdot v \in \R \forall v \in \tango{x}{\M_{1}}$ and $\tango{x}{\map_{1}}$ is a linear map. We can associate to $\tango{x}{\map_{1}}$ a covector $\eta \in \left(\tspace{x}{\M_{1}}\right)^*$ such that $\tango{x}{\map_{1}}(v) \coloneqq \inprod{\eta}{v} = \tango{x}{\map_{1}} \cdot v \forall v \in \tspace{x}{\M_{1}}$. Therefore, given a smooth map $\map_{2}: \M_{2} \to \M_{1}$, the expression $\tangostar{x}{\map_{2}}(\tango{\map_{2}(x)}{\map_{1}})$ makes perfect sense (for rigorous details see \cite[Proposition 11.18]{lee}). Note that the same ideas applies to derivatives of real-valued functions defined on vector spaces.  As we go forward in the paper, we insert helpful comments to explain intuition behind differential geometric concepts. 

Since our setting is that of a Lie group, let us quickly recall some preliminaries. Refer \cite[Chapter 5,6]{holm} for more details on Lie groups. 
For our problem we will consider a $\dimg$ dimensional matrix Lie group $\LG \subset \R^{\dimem \times \dimem}$ with identity element $\id$ and Lie algebra $\LA \subset \R^{\dimem \times \dimem}$ \cite[Chapter 5.1]{holm}. Let $\sigma: \R^{\dimg} \to \LA$ be an isomorphism. 
The dual of $\LA$ is $\LA^* \subset \R^{\dimem \times \dimem}$ and $\inprod{\eta}{v} = \tr(\eta^{\top} v) \forall \eta \in \LA^*, v \in \LA$, where $\tr(\cdot)$ denotes the trace, and $(\cdot)^{\top}$ denotes the transpose.
Let $\Phi: \LG \times \LG \to \LG$ denote the group multiplication, which by definition, is matrix multiplication for matrix Lie groups, that is, $\Phi(\gr_1,\gr_2) = \gr_1\gr_2.$  
Then $\Phi_{\gr}: \LG \to \LG \forall \gr \in \LG$ is a diffeomorphism, and $\Phi_{\gr_1}(\gr_2) = \gr_1 \gr_2$ \cite[Chapter 5.2]{holm}.
For any $\gr \in \LG$, the tangent space is characterised as $T_{\gr}\LG = T_{\id}\Phi_{\gr}(\LA)$.
The tangent map of $\Phi$ simplifies to matrix multiplication as well, that is, given $\gr_1,\gr_2 \in \LG$ and $v \in \LA$, $T_{\id}\Phi_{\gr_1} v = \gr_1 v$ and therefore, $T_{\gr_1}\Phi_{\gr_2}\left(T_{\id}\Phi_{\gr_1} v \right)= \gr_2\gr_1 v$. 
The exponential map is denoted as $\exp : \LA \to \LG$, which for matrix Lie groups is simply the matrix exponential \cite[Chapter 5.4]{holm}. 
Given $\gr \in \LG$, the $\exp$ map can be used to define smooth curve on $\LG$ passing through $\gr$ with tangent vector $T_{\id}\Phi_{\gr} \cdot v$ for $v \in \LA$ as $\R \ni s \mapsto \gr\exp(vs) \in \LG$. 
This can be conveniently used to find tangent maps of real valued functions, say $f : \LG \to \R$ as $\tango{\gr}{f} \cdot v = \left. \frac{d}{ds}\right\rvert_{s=0} f(\gr\exp(vs))$. 
The Adjoint action of $\LG$ on $\LA$ is \cite[Definition 6.40]{holm}
\begin{equation*}
\LG \times \LA \ni \left(\gr, v \right) \mapsto \Ad_{\gr}v \defn \left.\frac{d}{ds}\right|_{s=0} \gr \exp(s v)\gr^{-1}  \in \LA
\end{equation*}
For any $\gr \in \LG$, let $\Ad^*_{\gr}$ be the dual of $\Ad_{\gr}$.

\begin{remark}
Although we restrict attention to matrix Lie groups here, our derivations proceed in full generality first and then are made specific to matrix Lie groups, thus making it applicable to any finite dimensional Lie group.
\label{rem:LG-gen}
\end{remark}

For any $n \in \mathbb{W}$, $\upto{n} \defn \{0,1,\ldots,n\}$. Let $N \in \mathbb{W}$  be fixed, and will be referred to as the control horizon throughout. Before moving towards defining our problem, we briefly recall the discrete time Pontryagin maximum principle on matrix Lie groups \cite{karmvir-2018} since it plays a central role in our development. Consider a discrete time control system evolving on $\LG \times \R^{\dimg}$
\begin{subequations} 
\begin{align}
\gru_{k+1} &= \gru_k \kinu(\gru_k ,\velu_k) & \forall k \in \upto{N-1}, \label{eq:kinematics-u} \\
\velu_{k+1} &=  \dynu(\gru_k,\velu_k,\conu_k)  & \forall k \in \upto{N-1}, \label{eq:dynamics-u}
\end{align} 
\label{eq:system-u}
\end{subequations}
where 
\en{
\item $\gru_k \in \LG$ and $\velu_k \in \R^{\dimg}$, $\forall k \in \upto{N}$ are the states,  
\item $\kinu : \LG \cross \R^{\dimg} \ra \LG$ and $\dynu : \LG \cross \R^{\dimg} \cross \R^{\dimu}\ra \R^{\dimg}$ are smooth, 
\item $\conu_k \in \Uc _k \subset \R^{\dimu}$ is the control for all $k \in \upto{N-1}$. 
}
Consider the optimal control problem
\begin{subequations}
\label{eq:opt-prob-u}
\begin{alignat}{2}
 \min\limits_{\ch{\conu}} & \quad  && \Ju(\ch{\gru},\ch{\velu},\ch{\conu}) \defn  \sum\limits_{k=0}^{N-1}{\tilde{c}_k(\gru_k,\velu_k,\conu_k)} \notag \\ &&& \qquad \qquad + \tilde{c}_N(\gru_N,\velu_N) \\
\text{s.t. } &&& \text{system } \eqref{eq:system-u},\\
& & & \conu _k \in \Uc _k, \: \: \forall k \in \upto{N-1} \\ 
& & &\gru_0 = \bar{\gru}_0,\velu_0 = \bar{\velu}_0, 
\end{alignat}
\end{subequations} 
where the notation goes as,
\en{
\item $\LG \cross \R^{\dimg} \cross \R^{\dimu} \ni (\gru,\velu,\conu) \mapsto \tilde{c}_k(\gru,\velu,\conu) \in \R$, and $\LG \cross \R^{\dimg} \ni (\gru,\velu) \mapsto \tilde{c}_N(\gru,\velu) \in \R$ 
are suitable smooth functions $\forall k \in \upto{N-1}$, 
\item $\ch{\gru} \defn (\gru_0,\gru_1,\ldots,\gru_N) \in \overbrace{\LG \times \LG \times \cdots \times \LG}^{N+1 \text{ factors }}$ and $\ch{\velu} \defn (\velu_0,\velu_1,\ldots,\velu_N) \in \R^{\dimg(N+1)}$ are the state sequences,  
\item $\ch{\conu} \defn (\conu_0,\conu_1,\ldots,\conu_{N-1}) \in \R^{\dimu N}$ is the control sequence.  
}
Assume that there exists an open set $\Og \subset \LA$ such that $\exp$ map restricted to $\Og$, i.e. $\exp: \Og \to \exp(\Og)$ is a diffeomorphism and $\kinu \in \exp(\Og)$, and that $\Uc_k$ is convex $\forall k \in \upto{N-1}$. 

\begin{theorem}[{\cite[Theorem 5]{karmvir-2018}}]
For the optimization problem in \eqref{eq:opt-prob-u} let the optimal control sequence be $\opt{\ch{\conu}}$  respectively corresponding to which the state trajectory is $\opt{\ch{\gru}}$ and $\opt{\ch{\velu}}$. 
Define the Hamiltonian as (for $\nuu \in \R$) 
\begin{multline*}
\upto{N-1} \times \LA^* \times (\R^{\dimg})^* \times \LG \times \R^{\dimg} \times \R^{\dimu} \ni \\ (k,\zetau,\xiu,\gru,\velu,\conu)  \mapsto  
\hamc(k,\zetau,\xi,\gru,\velu,\conu) 
\defn \\ \nuu \tilde{c}_k(\gru,\velu,\conu)  + \inprod{\zetau}{\exp^{-1}(\kinu(\gru,\velu))}  + \inprod{\xiu}{\dynu(\gru,\velu,\conu)}.
\end{multline*}
For all $k \in \upto{N-1}$, there exist covectors  $\zetau^k \in \LA^*,\xiu^k \in (\R^{\dimg})^*, $
and define $\opt{\gamma}_k \defn (\zetau_k,\xiu_k,\opt{\gru}_k,\opt{\velu}_k,\opt{\conu}_k)$,
\begin{equation*} 
\rhou^k \defn 
\tangostar{\id}{\left( \exp^{-1}\circ\Phi_{((\opt{\gru}_{k-1})^{-1}\opt{\gru}_k)} \right)} (\zetau^k),
\end{equation*} 
which satisfy the following necessary conditions 
\en{
\item Optimal state dynamics ($\forall k \in \upto{N}$): 
\begin{align*}
\opt{\gru}_{k+1} &= \opt{\gru}_k \exp(\deri_{\zetau}\hamc(\opt{\gamma}_k)),  \\
\opt{\velu}_{k+1} &= \deri_{\xiu}\hamc(\opt{\gamma}_k); 
\end{align*}
\item Adjoint equations ($\forall k \in \upto{N-1}$): 
\begin{align*}
&\xiu^{k-1} = \deri_{\velu}\hamc(\opt{\gamma}_k), \\
&\rhou^{k-1}  = \Ad_{\exp(-\deri_{\zetau} \hamc(\opt{\gamma}_k))}^* \rhou^k + \tangostar{\id}{\Phi_{\opt{\gru}_k}}(\tang{\opt{\gamma}_k}{\gru}{\hamc} );
\end{align*}
\item Transversality relations: 
\begin{align*}
\xiu^{N-1} &= \nuu\deriv{\velu}{\tilde{c}_N}{\opt{\gru}_N,\opt{\velu}_N}, \\ 
\rhou^{N-1} &=  \nuu\tangostar{\id}{\Phi_{\opt{\gru}_N}}(\tang{(\opt{\gru}_N,\opt{\velu}_N) }{\gru}{\tilde{c}_N});  
\end{align*}
\item Hamiltonian non-positive gradient condition ($\forall k \in \upto{N-1}$): 
\begin{align*}
\inprod{\deri_{\conu}\hamc(\opt{\gamma}_{k-1})}{\tilde{\conu}_{k-1}}  \le 0 \forall \opt{\conu}_{k-1} + \tilde{\conu}_{k-1} \in \Uc_{k-1};
\end{align*}
\item Non-triviality: $\nuu \leq 0$ and if $\nuu = 0$ then at least one of the covectors is non-zero.
}
\label{th:CH-u}
\end{theorem}

\subsection{Problem Definition}
Consider a discrete time control system evolving on $\LG \times \R^{\dimg}$ as
\begin{subequations}
\begin{align}
\gr_{k+1} &= \gr_k \kin(\gr_k ,\vel_k) \: & \forall k \in \upto{N-1}, \label{eq:kinematics} \\
\vel_{k+1} &= \dyn(\gr_k,\vel_k,\con_k, \dist_k) \: & \forall k \in \upto{N-1}, \label{eq:dynamics} 
\end{align} 
\label{eq:system}
\end{subequations}
where
\enspl{Sys}{
\item $\gr_k \in \LG$ and $\vel_k \in \R^{\dimg}$ (equivalently, $\sigma(\vel_k) \in \LA$) $\forall k \in \upto{N}$ are the states,  \label{list:Sysi}
\item $\kin : \LG \cross \R^{\dimg} \ra \LG$ and $\dyn : \LG \cross \R^{\dimg} \cross \R^{\dimu} \cross \R^{\dimd} \ra \R^{\dimg}$ are smooth and respectively represent the kinematics and dynamics of the system, \label{list:Sysii}
\item $\con_k \in \Uc _k \subset \R^{\dimu}$ and $\dist_k \in \D _k \subset \R^{\dimd}$ respectively are the control and unmeasured external disturbance  for $k \in \upto{N-1}$. \label{list:Sysiii}
}
\begin{remark}
Physical systems (whether on Euclidean spaces or manifolds) are typically modelled as evolving in continuous time as differential equations. However, for implementation purposes one finds it suitable to represent them as discrete-time systems using. Discretising systems that evolve on manifolds is a non-trivial task since the discretisation should respect the manifold structure of the configuration space of the system. Discrete mechanics finds applications here, see, e.g., \cite{marsden-2001}. We will suppose that $\kin$ has been obtained using such an appropriate discretisation technique. 
\end{remark}
Our aim is to obtain necessary conditions satisfied by the optimiser of the following optimal control problem:
\begin{subequations}
\label{eq:opt-prob}
\begin{alignat}{2}
 \min\limits_{\ch{\con}} \max\limits_{\ch{\dist}} & \quad && \J(\ch{\gr},\ch{\vel},\ch{\con},\ch{d}) \defn  \sum\limits_{k=0}^{N-1}{c_k(\gr _k,\vel_k,\con _k,\dist _k)} \notag \\
 &&& \qquad \qquad \qquad \qquad \qquad  \quad + c_N(\gr _N,\vel_N) \label{eq:cost} \\
\text{s.t.} & && \text{system } \eqref{eq:system}, \\
& & & \con _k \in \Uc _k, \dist_ k \in \D _k \: \forall k \in \upto{N-1},  \\
& & &\gr_0 = \bar{\gr}_0,\vel_0 = \bar{\vel}_0, 
\end{alignat}
\end{subequations}
where the notation goes as,
\enspl{N}{
\item $\LG \cross \R^{\dimg} \cross \R^{\dimu} \ni (\gr,\vel,\con) \mapsto c_k(\gr,\vel,\con) \in \R$, and $\LG \cross \R^{\dimg} \ni (\gr,\vel) \mapsto c_N(\gr,\vel) \in \R$ 
are suitable smooth functions $\forall k \in \upto{N-1}$, \label{list:Ni}
\item $\ch{\gr} \defn (\gr_0,\gr_1,\ldots,\gr_N) \in \overbrace{\LG \times \LG \times \cdots \times \LG}^{N+1 \text{ factors }}$ and $\ch{\vel} \defn (\vel_0,\vel_1,\ldots,\vel_N) \in \R^{\dimg(N+1)}$ are the state sequences,  \label{list:Nii}
\item $\ch{\con} \defn (\con_0,\con_1,\ldots,\con_{N-1}) \in \R^{\dimu N}$ is the control sequence,  \label{list:Niii}
\item $\ch{\dist} \defn (\dist_0,\dist_1,\ldots,\dist_{N-1}) \in \R^{\dimd N}$ is the disturbance sequence.  \label{list:Niv}
}
\begin{assumption}
The following assumptions (taken from \cite[Assumption~1]{karmvir-2018}) are required to transfer the optimisation problem into a Euclidean setting and to obtain necessary conditions that the optimiser satisfies.
\en{
\item There exists an open set $\Og \subset \LA$ such that $\exp$ map restricted to $\Og$, i.e. $\exp: \Og \to \exp(\Og)$ is a diffeomorphism and the discretisation is such that $\kin \in \exp(\Og)$.
\item $\Uc _k$ and $\D _k$ are compact and convex sets $\forall k \in \upto{N}$. 
}
\label{assn:exp}
\end{assumption}

We assume that the cost in \eqref{eq:cost} admits a saddle point.

\begin{assumption}
$(\opt{\ch{\gr}},\opt{\ch{\vel}},\opt{\ch{\con}},\opt{\ch{\dist}})$ is a saddle point for the optimization problem \eqref{eq:opt-prob}.
\end{assumption}
Then defining open sets $\ch{\Uc} \subset \Uc _0 \times \Uc _1 \times \ldots \times \Uc _{N-1}$ and $\ch{\D} \subset \D _0 \times \D _1 \times \ldots \times \D _{N-1} $, with $\opt{\ch{\con}} \in \ch{\Uc}$ and $\opt{\ch{\dist}} \in \ch{\D}$,
\begin{align*}
&\J(\opt{\ch{\gr}},\opt{\ch{\vel}}, \opt{\ch{\con}}, \opt{\ch{\dist}})  \le \J(\ch{\gr},\ch{\vel}, \ch{\con}, \opt{\ch{\dist}}) \forall \ch{u} \in \ch{\Uc}, \\ 
&\J(\ch{\gr},\ch{\vel}, \opt{\ch{\con}}, {\ch{\dist}})  \le \J(\opt{\ch{\gr}},\opt{\ch{\vel}}, \opt{\ch{\con}}, \opt{\ch{\dist}}) \forall \ch{d} \in \ch{\D}.
\end{align*}

\subsection{Main Result}
\begin{theorem}
For the optimization problem in \eqref{eq:opt-prob} let the optimal control and disturbance sequence be $\opt{\ch{\con}}$ and $\opt{\ch{\dist}}$ respectively corresponding to which the state trajectory is $\opt{\ch{\gr}}$ and $\opt{\ch{\vel}}$. 
Define the Hamiltonian as 
\begin{multline*}
\upto{N-1} \times \LA^* \times (\R^{\dimg})^* \times \LG \times \R^{\dimg} \times \R^{\dimu} \times \R^{\dimd}  \ni \\ (k,\zeta,\xi,\gr,\vel,\con,\dist)  \mapsto 
\ham(k,\zeta,\xi,\gr,\vel,\con,\dist) \defn \\  -c_k(\gr,\vel,\con,\dist) + \inprod{\zeta}{\exp^{-1}(\kin(\gr,\vel))}  + \inprod{\xi}{\dyn(\gr,\vel,\con,{\dist})}.
\end{multline*}
For all $k \in \upto{N-1}$, there exist covectors  $\zeta^k \in \LA^*,\xi^k \in (\R^{\dimg})^*, $
and define 
$\opt{\gamma}_k \defn (\zeta_k,\xi_k,\opt{\gr}_k,\opt{\vel}_k,\opt{\con}_k,\opt{\con}_k)$, 
\begin{equation*} 
\rho^k \defn 
\tangostar{\id}{\left( \exp^{-1}\circ\Phi_{((\opt{\gr}_{k-1})^{-1}\opt{\gr}_k)} \right)} (\zeta^k),
\end{equation*} 
which satisfy the following necessary conditions
\enspl{H}{
\item Optimal state dynamics ($\forall k \in \upto{N}$): \label{list:CH-i}
\begin{align*}
\opt{\gr}_{k+1} &= \opt{\gr}_k \exp(\deri_{\zeta}\ham(\opt{\gamma}_k)),  \\
\opt{\vel}_{k+1} &= \deri_{\xi}\ham(\opt{\gamma}_k) ;
\end{align*}
\item Adjoint equations ($\forall k \in \upto{N-1}$): \label{list:CH-ii}
\begin{align*}
&\xi^{k-1} = \deri_{\vel}\ham(\opt{\gamma}_k), \\
&\rho^{k-1}  = \Ad_{\exp(-\deri_{\zeta} \ham(\opt{\gamma}_k))}^* \rho^k + \tangostar{\id}{\Phi_{\opt{\gr}_k}}(\tang{\opt{\gamma}_k}{\gr}{\ham} );
\end{align*}
\item Transversality relations: \label{list:CH-iii}
\begin{align*}
\xi^{N-1} &= -\deriv{\vel}{c_N}{\opt{\gr}_N,\opt{\vel}_N}, \\ 
\rho^{N-1} &=  -\tangostar{\id}{\Phi_{\opt{\gr}_N}}(\tang{(\opt{\gr}_N,\opt{\vel}_N) }{\gr}{c_N});  
\end{align*}
\item Hamiltonian ``saddle point" condition ($\forall k \in \upto{N-1}$): \label{list:CH-iv}
\begin{align*}
\inprod{\deri_{\con}\ham(\opt{\gamma}_{k-1})}{\tilde{u}_{k-1}}  \le 0 \forall \opt{u}_{k-1} + \tilde{u}_{k-1} \in \Uc_{k-1},& \\
\inprod{\deri_{\dist}\ham(\opt{\gamma}_{k-1})}{\tilde{d}_{k-1}}  \ge 0 \forall \opt{\dist}_{k-1} + \tilde{\dist}_{k-1} \in \D_{k-1}.&
\end{align*}
}
\label{th:CH}
\end{theorem}
\begin{proof}
Section \ref{sec:proof}.
\end{proof}

\begin{remark}
A comment on the notation used here, which is heavy in differential geometric jargon:
\en{
\item Condition \ref{list:CH-iv} does not strictly denote a saddle point in the most general case. However, under suitable assumptions on $\ham$ it does, for instance, if the Hamiltonian is convex in $\con$ and concave in $\dist$. 
\item The derivative of $\ham$ with respect to $\vel, \xi, \zeta$ are well understood since these entities lie in vector spaces.
\item Note that $\rho,\zeta$ can be written as elements of $\R^{\dimem \times \dimem}$. Since a co-vector can be easily visualised by its action on a vector, we will perform some calculations to make this apparent. To this end, let $v \in \LA$ be arbitrary and let $A \defn \exp(-\deri_{\zeta} \ham(\opt{\gamma}_k))$. To understand the relation between $\rho^{k-1}$ and $\rho^k$,
\begin{align*}
&\inprod{\rho^{k-1}}{v} = \inprod{\Ad_{A}^* \rho^k + \tangostar{\id}{\Phi_{\opt{\gr}_k}}(\tang{\opt{\gamma}_k }{\gr}{\ham} )}{v}  \\
&= \inprod{\rho^k}{\Ad_{A} v} + \inprod{\tang{\opt{\gamma}_k }{\gr}{\ham}}{\tango{\id}{\Phi_{\opt{\gr}_k}} \cdot v} \\
&=\tr\left((\rho^k)^{\top}(AvA)\right) + \inprod{\tang{\opt{\gamma}_k }{\gr}{\ham}}{\opt{\gr}_k  v} \\
&=\tr\left((\rho^k)^TAvA\right) + \left. \frac{d}{ds} \right\rvert_{s=0} \ham\left(\opt{\gr}_k \exp(vs)\right)
\end{align*}
To understand the relation between $\rho^k$ and $\zeta^k$
\begin{align*}
&\inprod{\rho^{k}}{v} = \inprod{\tangostar{\id}{\left( \exp^{-1}\circ\Phi_{((\opt{\gr}_{k-1})^{-1}\opt{\gr}_k)} \right)} (\zeta^k)}{v}  \\
&= \inprod{ \zeta^k}{\tango{\id}{\left( \exp^{-1}\circ\Phi_{((\opt{\gr}_{k-1})^{-1}\opt{\gr}_k)} \right)}\cdot v} \\
&=\inprod{\zeta^k}{\left. \frac{d}{ds} \right\rvert_{s=0}  \exp^{-1}\circ\Phi_{((\opt{\gr}_{k-1})^{-1}\opt{\gr}_k)} (\exp(vs))} \\
&=\inprod{\zeta^k}{\left. \frac{d}{ds} \right\rvert_{s=0}  \exp^{-1}\left((\opt{\gr}_{k-1})^{-1}\opt{\gr}_k \exp(vs)\right)} \\
&= \tr\left( (\zeta^k)^{\top}\left. \frac{d}{ds} \right\rvert_{s=0}  \exp^{-1}\left((\opt{\gr}_{k-1})^{-1}\opt{\gr}_k \exp(vs)\right) \right) 
\end{align*}
}
\end{remark}

\subsection{Specialisation to Euclidean Spaces}
Let the control system \eqref{eq:kinematics},\eqref{eq:dynamics} evolve on $\R^{\dimg}$, i.e. $\LG = \R^{\dimg}$ (although we require $\LG$ to be a matrix Lie group, as noted in Remark \ref{rem:LG-gen}, the part of our presentation made in full generality holds for any finite dimensional Lie group). We will introduce new notation for the cost and dynamics to avoid clashing with previous notation. Since the system evolves on a Euclidean space in entirety, we club all state variables into $\vel \in \R^n$ and kinematics and dynamics into  $\dyneuc : \R^{\dimg} \cross \R^{\dimu} \cross \R^{\dimd} \ra \R^{\dimg}$ .
For complete clarity, we shall restate the specific form of the control system
\begin{align}
\vel_{k+1} &= \dyneuc(\vel_k,\con_k, \dist_k) \: & \forall k \in \upto{N-1}, \label{eq:euc-system} 
\end{align} 
and optimal control problem
\begin{subequations}
\label{eq:opt-prob-euc}
\begin{alignat}{2}
 \min\limits_{\ch{\con}} \max\limits_{\ch{\dist}} & \quad  && \Jeuc(\ch{\vel},\ch{\con},\ch{d}) \defn  \sum\limits_{k=0}^{N-1}{\ceuc_k(\vel_k,\con _k,\dist _k)} \notag \\
 &&& \qquad \qquad \qquad \qquad \qquad + \ceuc_N(\vel_N) \\
\text{s.t. } & && \text{system } \eqref{eq:euc-system}, \\
& & & \con _k \in \Uc _k, \dist_ k \in \D _k  \:  \forall k \in \upto{N-1}, \\
& & &\vel_0 = \bar{\vel}_0, 
\end{alignat}
\end{subequations} 
where $\R^{\dimg} \cross \R^{\dimu} \ni (\vel,\con) \mapsto \ceuc_k(\vel,\con) \in \R$, and $\R^{\dimg} \ni \vel \mapsto \ceuc_N(\gr,\vel) \in \R$ 
are suitable smooth functions $\forall k \in \upto{N-1}$, and the remaining notation is same as List \ref{list:Sysi},\ref{list:Sysiii} and List \ref{list:Nii},\ref{list:Niii},\ref{list:Niv} (accommodating for the change of Lie group $\LG$ to $\R^n$).

\begin{corollary}
For the optimisation problem in \eqref{eq:opt-prob-euc} let the optimal control and disturbance sequence be $\opt{\ch{\con}}$ and $\opt{\ch{\dist}}$ respectively corresponding to which the state trajectory is $\opt{\ch{\gr}}$ and $\opt{\ch{\vel}}$. 
Define the Hamiltonian as 
\begin{multline*}
\upto{N-1} \times \R^{\dimg} \times \R^{\dimg} \times \R^{\dimu} \times \R^{\dimd} \ni  
(k,\xi,\vel,\con,\dist)  \mapsto \\ 
\ham(k,\xi,\vel,\con,\dist) \defn  
-\ceuc_k(\vel,\con)  + \inprod{\xi}{\dyneuc(\vel,\con,{\dist})}.
\end{multline*}
For $k \in \upto{N-1}$, there exist covectors  $\xi^k \in \R^{\dimg}, $
and define 
$\opt{\gamma}_k \defn (\xi_k,\opt{\vel}_k,\opt{\con}_k,\opt{\dist}_k)$, 
which satisfy the following necessary conditions
\en{
\item Optimal state dynamics ($\forall k \in \upto{N}$): 
\begin{align*}
\opt{\vel}_{k+1} &= \deri_{\xi}\ham(\opt{\gamma}_k);   
\end{align*}
\item Adjoint equations ($\forall k \in \upto{N-1}$): 
\begin{align*}
&\xi^{k-1} = \deri_{\vel}\ham(\opt{\gamma}_k); 
\end{align*}
\item Transversality relations: 
\begin{align*}
\xi^{N-1} &= -\deriv{\vel}{\ceuc_N}{\opt{\vel}_N}; 
\end{align*}
\item Hamiltonian ``saddle point condition" ($\forall k \in \upto{N-1}$): 
\begin{align*}
\inprod{\deri_{\con}\ham(\opt{\gamma}_{k-1})}{\tilde{u}_{k-1}}  \le 0 \forall \opt{u}_{k-1} + \tilde{u}_{k-1} \in \Uc_{k-1},& \\
\inprod{\deri_{\dist}\ham(\opt{\gamma}_{k-1})}{\tilde{d}_{k-1}}  \ge 0 \forall \opt{\dist}_{k-1} + \tilde{\dist}_{k-1} \in \D_{k-1}.&
\end{align*}
}
\end{corollary}

\begin{proof}
Direct application of Theorem \ref{th:CH} on optimisation problem \eqref{eq:opt-prob-euc} (noting en route the change of $\LG$ to $\R^{\dimg}$).
\end{proof}

\begin{remark}
We make some observations that simplify computations when $\LG = \R^{\dimg}$. The group multiplication is then vector space addition on $\R^{\dimg}$ i.e. $\Phi(\vel_1,\vel_2) = \vel_1 + \vel_2 \forall \vel_1,\vel_2 \in \R^{\dimg}$, the $\exp$ map is the identity map \cite[Chapter 9, Page 274]{marsden}, and it is clear from the definition of the adjoint action that it simplifies to the identity transformation as well. Tangent maps simplify to derivatives, and for any $\vel \in \R^{\dimg}$, $\tango{\vel}{\R^{\dimg}} \cong \R^{\dimg}$ (in particular, $\LA \cong \R^{\dimg}$) and $\left(\tango{\vel}{\R^{\dimg}}\right)^* \cong \R^{\dimg}$ (in particular, $\LA^* \cong \R^{\dimg}$). 
\end{remark}

\section{Proof of Theorem \ref{th:CH}}
\label{sec:proof}
The proof will be based on the methodology detailed 
in List \ref{list:SP1},\ref{list:SP2},\ref{list:SP3}. The sketch of proof is as follows: 
\enspl{{\it Step}}{
\item \label{list:Step-i} {\bf Minimisation:} Freeze $\ch{\dist}$ at $\opt{\ch{\dist}}$. Get necessary conditions on $\ch{\gr},\ch{\vel},\ch{\con}$ for minimisation over $\ch{\con}$ using Theorem \ref{th:CH-u}.
\item \label{list:Step-ii} {\bf Maximisation:} Freeze $\ch{\con}$ at $\opt{\ch{\con}}$. Get necessary conditions on  $\ch{\gr},\ch{\vel},\ch{\dist}$ for maximising over $\ch{\dist}$ using Theorem \ref{th:CH-u}.
\item \label{list:Step-iii} {\bf Amalgamation:} We will coherently merge the results from the two previous steps into a single set of necessary conditions, which will have a single Hamiltonian.
}

\subsection{\ref{list:Step-i} Minimisation}
Freezing $\ch{\dist}$ at $\opt{\ch{\dist}}$, the optimisation problem \eqref{eq:opt-prob} becomes
\begin{subequations}
\label{eq:opt-prob-con}
\begin{alignat}{2}
 \min\limits_{\ch{\con}} & \quad  && \J(\ch{\gr},\ch{\vel},\ch{\con},\opt{\ch{d}}) \\
\text{s.t.} & && \text{system } \eqref{eq:system} \text{ with } \dist_k = \opt{\dist}_k \forall k \in \upto{N-1}, \\
& & & \con _k \in \Uc _k \: \: \forall k \in \upto{N-1},  \\
& & &\gr_0 = \bar{\gr}_0,\vel_0 = \bar{\vel}_0. 
\end{alignat}
\end{subequations}

For the optimisation problem in \eqref{eq:opt-prob-con} let the optimal control sequence be $\opt{\ch{\con}}$ corresponding to which the state trajectory is $\opt{\ch{\gr}}$ and $\opt{\ch{\vel}}$.  Applying Theorem \ref{th:CH-u}, we define the Hamiltonian for the control case as (for $\check{\nu} \in \R$): 
\begin{align}
&\upto{N-1} \times \LA^* \times (\R^{\dimg})^* \times \LG \times \R^{\dimg} \times \R^{\dimu}  \ni (k,\check{\zeta},\check{\xi},\gr,\vel,\con) \notag \\  
&\mapsto \hamu^{\check{\nu}}(k,\check{\zeta},\check{\xi},\gr,\vel,\con) \defn \check{\nu} c_k(\gr,\vel,\con,\opt{\dist}_k) \notag \\ 
 & \qquad \quad + \inprod{\check{\zeta}}{\exp^{-1}(\kin(\gr,\vel))}  + \inprod{\check{\xi}}{\dyn(\gr,\vel,\con,\opt{\dist}_k)}.
\label{eq:ham-con}
\end{align}
For all $k \in \upto{N-1}$, there exist covectors $\check{\zeta}^k \in \LA^*$, $\check{\xi}^k \in (\R^{\dimg})^*$ and define $\opt{\check{\gamma}}_k \defn (\check{\zeta}_k,\check{\xi}_k,\opt{\gr}_k,\opt{\vel}_k,\opt{\con}_k)$, 
\begin{equation*}
\check{\rho}^k \defn  
\tangostar{\id}{\left( \exp^{-1}\circ\Phi_{((\opt{\gr}_{k-1})^{-1}\opt{\gr}_k)} \right)} (\check{\zeta}^k),
\end{equation*} 
which satisfy the following necessary conditions 
\enspl{$u$}{
\item Optimal state dynamics ($\forall k \in \upto{N}$): \label{list:m-i}
\begin{align*}
\opt{\gr}_{k+1} &= \opt{\gr}_k \exp\left(\deriv{\check{\zeta}}{\hamu^{\check{\nu}}}{\opt{\check{\gamma}}_k}\right),  \\
\opt{\vel}_{k+1} &= \deriv{\check{\xi}}{\hamu^{\check{\nu}}}{\opt{\check{\gamma}}_k} ;
\end{align*}
\item Adjoint equations ($\forall k \in \upto{N-1}$): \label{list:m-ii}
\begin{align}
\check{\xi}^{k-1} &= \deriv{\vel}{\hamu^{\check{\nu}}}{\opt{\check{\gamma}}_k}, \label{eq:covecprop-w-con} \\
\check{\rho}^{k-1}  &= \Ad_{\exp(-\deriv{\check{\zeta}}{\hamu^{\check{\nu}}}{\opt{\check{\gamma}}_k})}^* \check{\rho}^k \notag  \\ & \qquad \qquad + \tangostar{\id}{\Phi_{\opt{\gr}_k}}\Big(\tang{\opt{\check{\gamma}}_k }{\gr}{\hamu^{\check{\nu}}} \Big); \label{eq:covecprop-g-con}
\end{align}
\item Transversality relations: \label{list:m-iii}
\begin{align}
\check{\xi}^{N-1} &= \check{\nu}\deriv{\vel}{c_N}{\opt{\gr}_N,\opt{\vel}_N},  \label{eq:trans-w-con} \\
\check{\rho}^{N-1} &=  \check{\nu}\tangostar{\id}{\Phi_{\opt{\gr}_N}}(\tang{(\opt{\gr}_N,\opt{\vel}_N )}{\gr}{c_N}); \label{eq:trans-g-con}
\end{align}
\item Hamiltonian non-positive gradient condition ($\forall k \in \upto{N-1}$): \label{list:m-iv}
\begin{align*}
\inprod{\deriv{\con}{\hamu^{\check{\nu}}}{\opt{\check{\gamma}}_{k-1}}}{\tilde{\con}_{k-1}} \le 0 \forall \opt{\con}_{k-1} + \tilde{\con}_{k-1} \in \Uc_{k-1};
\end{align*}
\item Non-triviality: $\check{\nu} \leq 0$ and if $\check{\nu} = 0$ then at least one of the covectors is non-zero. \label{list:m-v}
}

\subsection{\ref{list:Step-ii} Maximisation}
Having dealt with the minimisation problem, we move on to the maximisation problem. Freezing $\ch{\con}$ at $\opt{\ch{\con}}$, the maximisation problem corresponding to \eqref{eq:opt-prob} is
\begin{subequations}
\label{eq:opt-prob-dist}
\begin{alignat}{2}
 \min\limits_{\ch{\dist}} & \quad  && -\J(\ch{\gr},\ch{\vel},\opt{\ch{\con}},\ch{d}) \\
\text{s.t. } & && \text{system } \eqref{eq:system} \text{ with } \con_k = \opt{\con}_k \forall k \in \upto{N-1}, \\
& & & \dist_ k \in \D _k  \: \: \forall k \in \upto{N-1},  \\
& & &\gr_0 = \bar{\gr}_0,\vel_0 = \bar{\vel}_0. 
\end{alignat}
\end{subequations}

For the optimisation problem in \eqref{eq:opt-prob-dist} let the optimal disturbance sequence be $\opt{\ch{\dist}}$ corresponding to which the state trajectory is $\opt{\ch{\gr}}$ and $\opt{\ch{\vel}}$.  Applying Theorem \ref{th:CH-u}, we define the Hamiltonian for the disturbance case as (for $\hat{\nu} \in \R$): 
\begin{align}
&\upto{N-1} \times (\LA)^* \times (\R^{\dimg})^* \times \LG \times \R^{\dimg} \times \R^{\dimd} \ni (k,\hat{\zeta},\hat{\xi},\gr,\vel,\dist) \notag \\ & \mapsto   
\hamd^{\hat{\nu}}(k,\hat{\zeta},\hat{\xi},\gr,\vel,\dist) \defn -\hat{\nu} c_k(\gr,\vel,\opt{\con},\dist) \notag \\ 
& \qquad \quad +  
\inprod{\hat{\zeta}}{\exp^{-1}(\kin(\gr,\vel)}   + \inprod{\hat{\xi}}{\dyn(\gr,\vel,\opt{\con}_k,{\dist})}.
\label{eq:ham-dist}
\end{align}

For all $k \in \upto{N-1}$ there exist covectors $\hat{\zeta}^k \in \LA^*$, $\hat{\xi}^k \in (\R^{\dimg})^*$, and define $\opt{\hat{\gamma}}_k \defn (\hat{\zeta}_k,\hat{\xi}_k,\opt{\gr}_k,\opt{\vel}_k,\opt{\dist}_k)$, 
\begin{equation*}
\hat{\rho}^k \defn 
\tangostar{\id}{\left( \exp^{-1}\circ\Phi_{((\opt{\gr}_{k-1})^{-1}\opt{\gr}_k)} \right)} (\hat{\zeta}^k),
\end{equation*}
which satisfy the following necessary conditions 
\enspl{$d$}{
\item Optimal state dynamics ($\forall k \in \upto{N}$):
\begin{align*}
\opt{\gr}_{k+1} &= \opt{\gr}_k \exp(\deriv{\zeta}{\hamd^{\hat{\nu}}}{\opt{\hat{\gamma}}_k}),  \\
\opt{\vel}_{k+1} &= \deriv{\xi}{\hamd^{\hat{\nu}}}{\opt{\hat{\gamma}}_k}; 
\end{align*}
\item Adjoint equations ($\forall k \in \upto{N-1}$): \label{list:M-ii}
\begin{align}
\hat{\xi}^{k-1} &= \deriv{\vel}{\hamd^{\hat{\nu}}}{\opt{\hat{\gamma}}_k},  \label{eq:covecprop-w-dist} \\
\hat{\rho}^{k-1}  &= \Ad_{\exp(-\deriv{\hat{\zeta}} {\hamd^{\hat{\nu}}}{\opt{\gamma}_k})}^* \rho^k \notag \\ & \qquad \qquad + \tangostar{\id}{\Phi_{\opt{\gr}_k}}\Big(\tang{\opt{\gamma}_k }{\gr}{ \hamd} \Big); \label{eq:covecprop-g-dist}
\end{align}
\item Transversality conditions: \label{list:M-iii}
\begin{align}
\hat{\xi}^{N-1} &= -\hat{\nu}\deriv{\vel}{{c}_N}{\opt{\gr}_N,\opt{\vel}_N},  \label{eq:trans-w-dist} \\
\hat{\rho}^{N-1} &=  -\hat{\nu}\tangostar{\id}{\Phi_{\opt{\gr}_N}}(\tang{(\opt{\gr}_N,\opt{\vel}_N )}{\gr}{{c}_N});  \label{eq:trans-g-dist}
\end{align}
\item Hamiltonian non-positive gradient condition ($\forall k \in \upto{N-1}$): \label{list:M-iv}
\begin{align*}
\inprod{\deriv{\dist}{\hamd^{\hat{\nu}}}{\opt{\hat{\gamma}}_{k-1}}}{\tilde{d}_{k-1}} \le 0 \forall \opt{\dist}_{k-1} + \tilde{\dist}_{k-1} \in \D_{k-1};
\end{align*}
\item Non-triviality: $\hat{\nu} \leq 0$ and if $\hat{\nu} = 0$ then at least one of the covectors is non-zero. \label{list:M-v}
}

\subsection{\ref{list:Step-iii} Amalgamation}

Up until now, we have two Hamiltonians $\hamu^{\check{\nu}}$ and $\hamd^{\hat{\nu}}$, and two sequences of covectors per Hamiltonian. The equations for the control case were parametrised by the optimal disturbance sequence and vice versa which makes the problem extremely hard to solve since they are two two-point boundary value problems paramatrised by the solutions of each other. However, as we will show, there emerges a single Hamiltonian and a single set of covectors which begins to suggest that the problem could be tractable. We establish this through the following steps 
\begin{itemize}
    \item Show that abnormal multipliers are non-zero
    \item Define the common Hamiltonian
    \item Covectors of both problems are negatives of each other, and can be obtained from the common Hamiltonian
    \item Obtain the ``saddle point" condition on the common Hamiltonian
    \end{itemize}

We begin with an observation which we will use often.
\begin{remark}
For all $k \in \upto{N-1}$, $\check{\zeta}^{k}$ and $\hat{\zeta}^{k}$ are related to $\check{\rho}^{k}$ and $\hat{\rho}^{k}$  through bijective linear maps, and in particular, $\check{\zeta}^{k} = 0 \iff \check{\rho}^{k} = 0$ and $\hat{\zeta}^{k} = 0 \iff \hat{\rho}^{k} = 0$.
\label{rem:bij-lin}
\end{remark}

\begin{proposition}
If $\check{\nu} = 0$, then $\check{\zeta}^{k} = 0$ and  $\check{\xi}^{k} = 0$ $\forall k \in \upto{N-1}$.
\label{prop:zero}
\end{proposition}

\begin{proof}
This proof will utilise mathematical induction.

{\it Base case}: If $\check{\nu} = 0$ then   \eqref{eq:trans-g-con}, \eqref{eq:trans-w-con} and Remark \ref{rem:bij-lin} yield $\check{\zeta}^{N-1} = 0$ and $\check{\xi}^{N-1} =0$ respectively. 

{\it Induction hypothesis}: Assume that the proposition holds for $k = N-1,N-2,\ldots,i$ for any $i \in \{ N-2, N-3, \ldots, 1\}$. We will show it also holds for $k = i-1$. 

{\it Induction step}: The induction hypothesis makes 
\begin{gather*}
\deriv{\vel}{\hamu^{\check{\nu}}}{\opt{\check{\gamma}}_{i}} = 0 , \quad \tang{\opt{\check{\gamma}}_{i}}{\gr} {\hamu^{\check{\nu}}} = 0.
\end{gather*}
Therefore from \eqref{eq:covecprop-g-con}, \eqref{eq:covecprop-w-con} and Remark \ref{rem:bij-lin},   $\zeta^{i-1} = 0$ and $\xi^{i-1} = 0$ respectively.
This completes the induction.
\end{proof}

The foregoing result also shows that $\check{\nu} = 0$ is not allowed since it violates the non-trviality condition \ref{list:m-v}. It follows that $\check{\nu} < 0$. It can similarly be shown that $\hat{\nu} < 0$. For convenience, we let $\check{\nu}=\hat{\nu}=-1$ and scale the covectors in both problems accordingly. See Appendix \ref{app:scaling} for details of scaling $\check{\nu}$, and a similar procedure follows for $\hat{\nu}$.
Define the common Hamiltonian
\begin{align*}
&\upto{N-1} \times \R^{\dimg} \times \R^{\dimg} \times \LG \times \R^{\dimg} \times \R^{\dimu} \times \R^{\dimd} \ni \\ 
&(k,\zeta,\xi,\gr,\vel,\con,\dist) \mapsto  
\ham(k,\zeta,\xi,\gr,\vel,\con,\dist)  
\defn \\&  -c_k(\gr,\vel,\con)  + \inprod{\zeta}{\kin(\gr,\vel)}  + \inprod{\xi}{\dyn(\gr,\vel,\con,{\dist})}.
\end{align*}
We will proceed to establish conditions \ref{list:CH-ii}, \ref{list:CH-iii}.
\begin{proposition}
For all $k \in \upto{N-1}$
\begin{enumerate}[label = {\rm (\roman*)}]
\item $\check{\zeta}^{k} = -\hat{\zeta}^{k}$ and $\check{\xi}^k = -\hat{\xi}^k$ $\forall k \in \upto{N-1}$
\item There exist covectors  $\zeta^k ,\xi^k \in \R^{\dimg}$ and 
\begin{equation*} 
\rho^k \defn 
\tangostar{\id}{\left( \exp^{-1}\circ\Phi_{((\opt{\gr}_{k-1})^{-1}\opt{\gr}_k)} \right)} (\zeta^k)
\end{equation*}
which satisfy conditions \ref{list:CH-ii}, \ref{list:CH-iii}.
\end{enumerate}
\end{proposition}

\begin{proof}
Let us begin by proving the result for $k = N-1$. Observe from   \eqref{eq:trans-g-con},\eqref{eq:trans-g-dist} and \eqref{eq:trans-w-con},\eqref{eq:trans-w-dist} that
\begin{gather*}
\check{\rho}^{N-1} = -\hat{\rho}^{N-1} = -\tangostar{\id}{\Phi_{\opt{\gr}_N}}(\tang{(\opt{\gr}_N,\opt{\vel}_N )}{\gr}{c_N}), \\
\check{\xi}^{N-1} = -\hat{\xi}^{N-1} = -\deriv{\vel}{c_N}{\opt{\gr}_N,\opt{\vel}_N},
\end{gather*}
and from Remark \ref{rem:bij-lin} ,$\zeta^{N-1} = -\hat{\zeta}^{N-1}$. Define $\xi^{N-1} \defn \check{\xi}^{N-1}$ and $\zeta^{N-1} \defn \check{\zeta}^{N-1}$ to finish the proof for $k = N-1$. 

For $k = N-2, N-3, \ldots, 0$ we use mathematical induction.

{\it Base case}: From the result for $k=N-1$ we obtain 
\begin{gather*}
\deriv{\vel}{\hamu^{\check{\nu}}}{\opt{\check{\gamma}}_{N-1}} =
 - \deriv{\vel}{\hamd^{\hat{\nu}}}{\opt{\hat{\gamma}}_{N-1}} =
\deriv{\vel}{\ham}{\opt{\gamma}_{N-1}}, \\
\deriv{\check{\zeta}}{\hamu^{\check{\nu}}}{\opt{\check{\gamma}}_{N-1}} 
=\deriv{\hat{\zeta}}{\hamd^{\hat{\nu}}}{\opt{\hat{\gamma}}_{N-1}} 
= \deriv{\zeta}{\ham}{\opt{{\gamma}}_{N-1}}, \\
 \tang{\opt{\check{\gamma}}_{N-1}}{\gr} {\hamu^{\check{\nu}}}
=-\tang{\opt{\hat{\gamma}}_{N-1}}{\gr}{\hamd^{\hat{\nu}}}
= \tang{\opt{\gamma}_{N-1}}{\gr} {\ham}. 
\end{gather*} 

Thus from  \eqref{eq:covecprop-g-con},\eqref{eq:covecprop-g-dist}, and \eqref{eq:covecprop-w-con},\eqref{eq:covecprop-w-dist}, observe that
\begin{align*} 
\check{\rho}^{N-2} &= -\hat{\rho}^{N-2} \\ &= \Ad_{\exp(-\deriv{\zeta}{ \ham^{\nu}}{\opt{\gamma}_{N-1}})}^* \check{\rho}^{N-1} \\ & \qquad + \tango{\id}{\Phi_{\opt{\gr}_{N-1}}}(\tang{(\opt{\gamma}_{N-1} )}{\gr}{\ham}), \\
\check{\xi}^{N-2} &= -\hat{\xi}^{N-2} = \deriv{\vel}{\ham}{\opt{\gamma}_{N-2}},
\end{align*}
and from Remark \ref{rem:bij-lin}, $\zeta^{N-2} = -\hat{\zeta}^{N-2}$. Define $\xi^{N-2} \defn \check{\xi}^{N-2}$ and $\zeta^{N-2} \defn \check{\zeta}^{N-2}$ to finish the proof for $k = N-2$.

{\it Induction hypothesis}: Assume that the condition \ref{list:CH-ii} holds for $k = N-2,N-3,\ldots,i$ for any $i \in \{ N-2, N-3, \ldots, 1\}$. We will show it also holds for $k = i-1$.

{\it Induction step}: Using an argument parallel to the one used for the base case, the induction step can be shown. 
\end{proof}

\begin{remark}
An alternate proof can be given by first making the observation that the covectors depend linearly on $\check{\nu}$ (or $\hat{\nu}$, as the case may be), and then absorbing the negative sign accompanying $c_k$ and $c_N$ in the maximisation case into $\hat{\nu}$.
\end{remark}

To obtain the ``saddle point'' condition \ref{list:CH-iv}, notice that the preceding induction argument also shows that
$\deriv{\con}{\hamu^{\check{\nu}}}{\opt{\check{\gamma}}_k} = \deriv{\con}{\ham}{\opt{\gamma}_k}$ and $\deriv{\dist}{\hamd^{\hat{\nu}}}{\opt{\hat{\gamma}}_k} = -\deriv{\dist}{\ham}{\opt{\gamma}_k}$ for $k \in \upto{N-1}$. Hence for $k \in \upto{N-1}$ we have,
\begin{align*}
\inprod{\deriv{\con}{\hamu^{\check{\nu}}}{\opt{\check{\gamma}}_k}}{\tilde{u}_k} & \le 0 \forall \opt{u}_k + \tilde{u}_k \in \Uc_k \\
\implies \inprod{\deriv{\con}{\ham}{\opt{\gamma}_k}}{\tilde{u}_k} & \le 0 \forall \opt{u}_k + \tilde{u}_k \in \Uc_k
\end{align*}
and
\begin{align*}
\inprod{\deriv{\dist}{\hamd^{\hat{\nu}}}{\opt{\gamma}_k}}{\tilde{d}_k} & \le 0 \forall \opt{\dist}_k + \tilde{\dist}_k \in \D_k \\
\implies \inprod{\deriv{\dist}{\ham}{\opt{\gamma}_k}}{\tilde{d}_k} & \ge 0 \forall \opt{\dist}_k + \tilde{\dist}_k \in \D_k
\end{align*}
Condition \ref{list:CH-i} is obtained by direct calculation.
\section{Example}
\label{sec:example}

To demonstrate the application of Theorem \ref{th:CH}, we use a simple example of single axis rotation of a spacecraft under bounded disturbance. 
It is essentially a continuation of the example presented in \cite{karmvir-2018}, to keep the spirit of extension of that work intact. 
Single axis rotation of a spacecraft evolves on $\SOtwo$. 
This is one of the easiest and most intuitive examples of a Lie group, which still preserves the structure intrinsic to the problem. 
Consider the Lie group $\SOtwo$ and its Lie algebra $\sotwo$,  
\begin{gather*}
\SOtwo \defn \set{\gr \in \R^{2 \times 2}}{\gr^{\top}\gr = \gr\gr^{\top} = \id, \det(\gr) = 1} \\
\sotwo \defn \set{v \in \R^{2 \times 2}}{v^{\top} + v = 0}
\end{gather*}

Define an isomorphism $\sigma$, its inverse $\vex$, and the $\exp$ map 
\begin{gather*}
\R \ni x \mapsto \sigma(x) \coloneqq \begin{pmatrix} 0 & -x \\ x & 0 \end{pmatrix} \in \sotwo, \: \vex(\begin{pmatrix} 0 & -x \\ x & 0 \end{pmatrix}) = x \\
\sotwo \ni \sigma(x) \mapsto \exp(\sigma(x)) \coloneqq \begin{pmatrix} \cos(x) & -\sin(x) \\ \sin(x) & \cos(x) \end{pmatrix}
\end{gather*}
By virtue of the Riesz representation theorem, we identify $(\sotwo)^*$ with $\sotwo$ and $\sigma^*$ with $\vex$.
For this Section, $\inprod{\eta}{v} = \frac{1}{2}\tr(\eta^{\top} v) \forall \eta \in (\sotwo)^*, v \in \sotwo$.
\begin{remark}
Any element $\gr \in \SOtwo$ can be represented as $\gr = \exp(\theta)$ for some $\theta \in [0,2\pi)$. 
\label{rem:SOtwo}
\end{remark}

We model single axis rotation as
\begin{subequations}
\begin{align}
\gr_{k+1} &= \gr_{k}\kinex(\gr_k,\vel_k) & \forall k \in \upto{N-1}, \label{eq:ex-sys-kin} \\
\vel_{k+1} &= \dynex(\gr_k,\vel_k,\con_k,\dist_k) & \forall k \in \upto{N-1} \label{eq:ex-sys-dyn}
\end{align}
\label{eq:ex-sys}
\end{subequations}
where
\en{
\item $\gr_k \in \SOtwo$, $\vel_k \in \R$ $\forall k \in \upto{N}$, 
\item $\Uc \defn [-\con_c,\con_c]$, $\D \defn [-\dist_c,\dist_c]$ and $\con_c,\dist_c,\step \in \R_{>0}$ are fixed, and initial conditions $\gr_0 = \bar{\gr}_0$, $\vel_0 = \bar{\vel}_0$ are known,
\item 
\begin{align*}
&\SOtwo \times \R \ni (\gr,\vel) \mapsto \kinex(\gr,\vel) \defn \\ & \qquad \qquad \begin{pmatrix} \sqrt{1-\step^2 \vel^2} & - \step \vel \\  \step \vel & \sqrt{1-\step^2 \vel^2} \end{pmatrix} \in \SOtwo,
\end{align*}
\item $\SOtwo \times \R \times \R \times \R \ni (\gr,\vel,\con,\dist) \mapsto \dynex(\gr,\vel,\con,\dist) \defn \vel + \step (\con + \dist). $
}

The optimal control problem is
\begin{subequations}
\label{eq:opt-prob-ex} 
\begin{alignat}{2}
 \min\limits_{\ch{\con}} \max\limits_{\ch{\dist}} & \quad  &&  \frac{1}{2}\bigg(\Lambda^2\vel_N^2 + \sum\limits_{k=0}^{N-1}\lambda^2\con_k^2 + \Lambda^2\vel_k^2 - \mu^2\dist_k^2 \notag \\
 &&& \qquad \qquad \qquad +  \psi^2(2 - \tr(\gr_k)) \bigg)    \label{eq:opt-prob-ex-cost} \\
\text{s.t. } & && \text{system } \eqref{eq:ex-sys}, \\
& & & \con _k \in \Uc _k, \dist_ k \in \D _k \: \: \forall k \in \upto{N-1}, \\
& & &\gr_0 = \bar{\gr}_0,\vel_0 = \bar{\vel}_0. 
\end{alignat}
\end{subequations}
\begin{remark}
The modelling interpretation of cost function is as follows.
$\lambda^2\con_k^2$ minimises control effort and $\Lambda^2\vel_k^2$ maintain low angular velocity.
$- \mu^2\dist_k^2$ penalises the effort exerted by the external disturbance, the justification for which is that Nature also tries to exert minimum effort.
The last term in the cost function, $(2 - \tr(\gr_k))$ penalises the deviation of $\gr_k$ from $\id$. To see this, utilise Remark \ref{rem:SOtwo} and let $\gr_k = \exp(\theta_k)$ for some $\theta_k \in [0,2\pi)$. Therefore,
\begin{align*}
2 - \tr(\gr_k) = 2(1 - \cos(\theta_k)) = 4\sin^2(\frac{\theta_k}{2}).
\end{align*} 
The quantity attains its minimum at $\theta_k = 0$ which is $\gr_k = \id$ which can be visualised in Fig. \ref{fig:group-cost}.
\end{remark}

\begin{figure}
\centering
\includegraphics[scale=0.5]{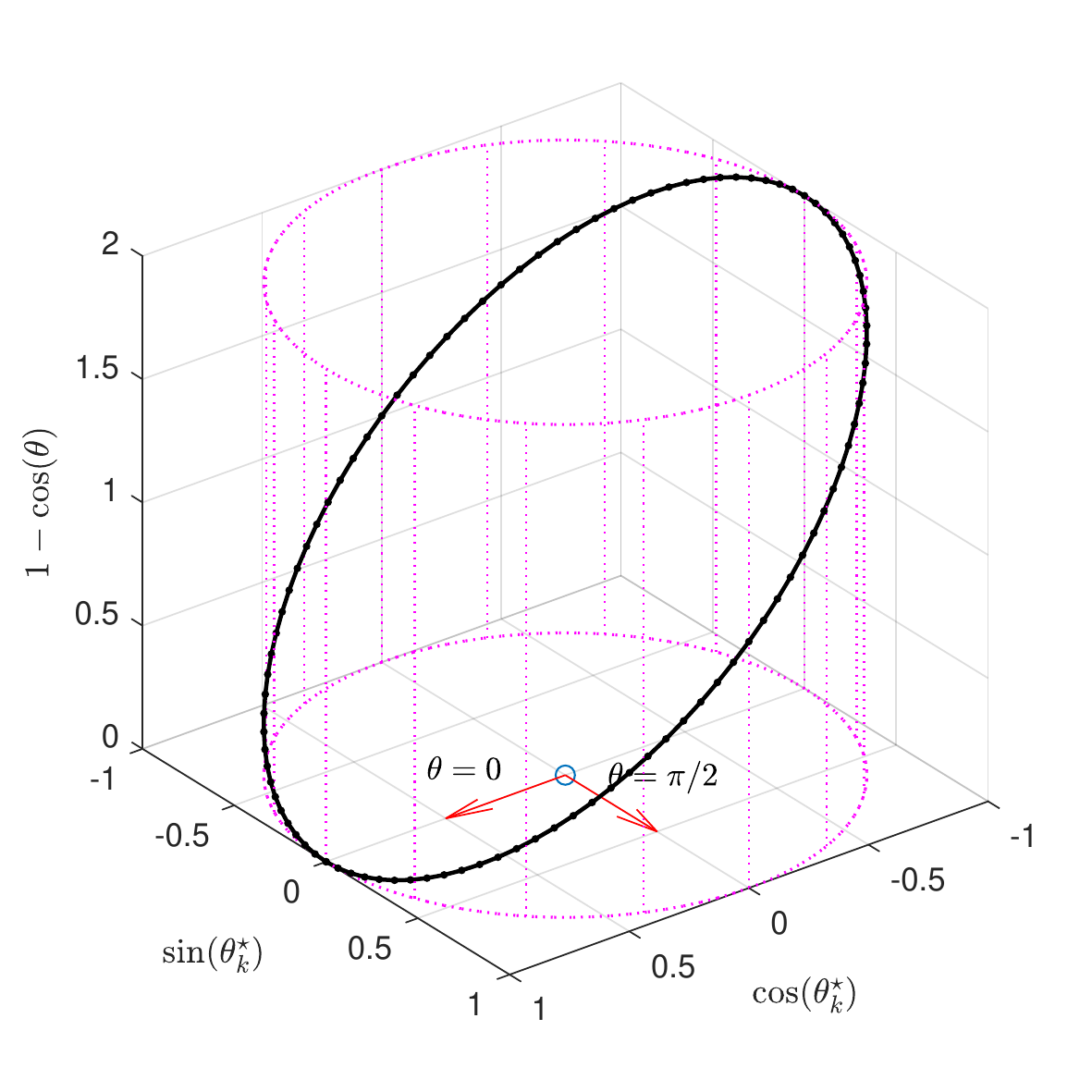}
\caption{Variation of $(2 - \tr(\gr_k))$ in the cost function}
\label{fig:group-cost}
\end{figure}
Assume that $\mu,\lambda,\Lambda,\psi \in \R_{>0}$ are carefully chosen so as to ensure that the cost function in \eqref{eq:opt-prob-ex-cost} admits a saddle point\footnote{In the simulations shown later, we verify numerically that the optimiser obtained is indeed a saddle point.}.
For the optimisation problem in \eqref{eq:opt-prob-ex} let the optimal control and disturbance sequence be $\opt{\ch{\con}}$ and $\opt{\ch{\dist}}$ respectively, corresponding to which the state trajectory is $\opt{\ch{\gr}}$ and $\opt{\ch{\vel}}$.  Applying Theorem \ref{th:CH}, we define the Hamiltonian (which is independent of the time instant since the cost, kinematics and dynamics are independent of time instant) as:
\begin{align*}
&(\sotwo)^* \times \R^* \times \SOtwo \times \R \times \R \times \R \ni (\zeta_{\sigma},\xi,\gr,\vel,\con,\dist)  \mapsto \notag \\ 
&\ham(\zeta_{\sigma},\xi,\gr,\vel,\con,\dist)  \defn -\left(\frac{\lambda^2\con^2 + \Lambda^2\vel^2  - \mu^2\dist^2 - \psi^2 \tr(\gr)}{2} \right) \\ & + \inprod{\zeta_{\sigma}}{\exp^{-1}(\kinex(\gr,\vel))}  + \xi \dynex(\gr,\vel,\con,\dist) \notag \\
 &= -\left(\frac{\lambda^2\con^2 + \Lambda^2\vel^2 - \mu^2\dist^2 -\psi^2 \tr(\gr)}{2} \right) \\ 
 &+ \zeta\sin^{-1}(\step\vel)+ \xi(\vel + \step (\con + \dist))
\end{align*}
where $\zeta \defn \sigma^{*}(\zeta_{\sigma}) = \vex(\zeta_{\sigma})$. The optimal state dynamics and adjoint dynamics are obtained as(refer Appenndix \ref{app:ex-two-adj} for details) 
\begin{align}
\opt{\gr}_{k+1} &= \opt{\gr}_{k}\kinex(\opt{\gr}_k,\opt{\vel}_k), \label{eq:ocp2-group-dyn} \\
\opt{\vel}_{k+1} &= \dynex(\opt{\gr}_k,\opt{\vel}_k,\opt{\con}_k,\opt{\dist}_k), \label{eq:ocp2-vel-dyn} \\
\zeta^{k-1}_{\sigma} & = \Ad_{\exp(-\deriv{\zeta_{\sigma}}{\ham}{\opt{\gamma}_k})}^* \zeta^k_{\sigma} +  \tangostar{\id}{\Phi_{\opt{\gr}_k}}(\tang{\opt{\gamma}_k}{\gr}{\ham} ), \notag\\
&= 
\kinex(\opt{\gr}_k,\opt{\vel}_k)\zeta^k_{\sigma}\left(\kinex(\opt{\gr}_k,\opt{\vel}_k)\right)^{\top} + \tangostar{\id}{\Phi_{\opt{\gr}_k}}(\tang{\opt{\gamma}_k}{\gr}{\ham} ), \notag \\ 
&= \zeta^k_{\sigma} + \psi^2 \frac{(\opt{\gr}_k)^{\top} - \opt{\gr}_k}{2}, \label{eq:ocp2-zeta-dyn}\\
 \xi^{k-1} &= \deriv{\vel}{ \ham}{\opt{\gamma}_k} = \frac{\step\zeta^k }{\sqrt{1-(\step \opt{\vel}_k)^2}} + \xi^{k} - \Lambda^2\opt{\vel}_k, \notag \\
\zeta^{N-1}_{\sigma}& = -\tangostar{\id}{\Phi_{\opt{\gr}_N}}(\tang{(\opt{\gr}_N,\opt{\vel}_N )}{\gr}{c_N}) = \psi^2 \frac{(\opt{\gr}_N)^{\top} - \opt{\gr}_N}{2}, \label{eq:ocp2-zeta-trans} \\
\xi^{N-1} &= -\deriv{\vel}{c_N}{\opt{\gr}_N,\opt{\vel}_N} = -\Lambda^2\opt{\vel}_N. \notag
\end{align}
The adjoint dynamics simplify to the following for $k \in \upto{N-1}$
\begin{align}
\zeta_{\sigma}^k &= \sum_{i=k+1}^{N} \psi^2\frac{(\opt{\gr}_i)^{\top} - \opt{\gr}_i}{2}, \label{eq:ocp2-zeta} \\
\xi^k &= \step\sum_{i=k+1}^{N-1} \frac{1}{\sqrt{1-(\step \opt{\vel}_k)^2}}\sum_{j=i+1}^{N}\psi^2\vex\left(\frac{(\opt{\gr}_j)^{\top} - \opt{\gr}_j}{2} \right) \notag \\ 
& \qquad \qquad -\sum_{i=k+1}^{N}\Lambda^2\opt{\vel}_i. \label{eq:ocp2-xi}
\end{align}
Since the Hamiltonian is convex in $\con$ and concave in $\dist$, the optimal control and disturbance are easily derived as
\begin{align}
\opt{\con}_{k}& = \argmax\limits_{w \in [-\con_c,\con_c]} \ham ({\zeta}^k_{\sigma},\xi^k,\opt{\gr}_{k}, \opt{\vel}_k, w,\opt{\dist}_k ), \notag \\
&= \begin{cases}
\con_c \quad &\text{if} \; \dfrac{\step \xi^k}{\lambda^2} \geq \con_c \\
-\con_c \quad &\text{if} \; \dfrac{\step \xi^k}{\lambda^2} \leq -\con_c \\
\dfrac{\step \xi^k}{\lambda^2} \quad & \text{otherwise}
\end{cases}
\label{eq:ocp2-control}
\end{align}
\begin{align}
\opt{\dist}_{k}& = \argmin\limits_{w \in [-\dist_c,\dist_c]} \ham ({\zeta}^k_{\sigma},\xi^k,\opt{\gr}_{k}, \opt{\vel}_k, \opt{\con}_k, w ), \notag \\
& = \begin{cases}
\dist_c \quad &\text{if} \; \dfrac{\step \xi^k}{\mu^2} \leq -\dist_c \\
-\dist_c \quad &\text{if} \; \dfrac{\step \xi^k}{\mu^2} \geq \dist_c \\
-\dfrac{\step \xi^k}{\mu^2} \quad & \text{otherwise}
\end{cases}
\label{eq:ocp2-disturbance}
\end{align}

\begin{remark}
If in \eqref{eq:opt-prob-ex}, we let $\psi = 0$, the problem reduces to a linear quadratic dynamic game (which is a robust linear system with quadratic cost). This is because $(\gr_k)_{k=1}^{N}$ doesn't appear in  \eqref{eq:opt-prob-ex-cost}, hence the constraint imposed by \eqref{eq:ex-sys-kin} can be trivially satisfied, and \eqref{eq:ex-sys-dyn} and  \eqref{eq:opt-prob-ex-cost} define a linear system with quadratic cost. Additionally, if we let the control and disturbance be unconstrained, the solution of the resulting problem is well known \cite[Chapter 3.2.1]{basar-1995}, and it satisfies the necessary conditions given by Theorem \ref{th:CH}. See Appendix \ref{app:lq-dyn-game} for details.
\end{remark}

If we substitute \eqref{eq:ocp2-zeta} and \eqref{eq:ocp2-xi}, into \eqref{eq:ocp2-control} and \eqref{eq:ocp2-disturbance}, and substitute the resulting equations into \eqref{eq:ocp2-group-dyn} and \eqref{eq:ocp2-vel-dyn}, we effectively change the necessary conditions a set of non-linear equations in $\left\{(\opt{\gr}_k,\opt{\vel}_k)\right\}_{k=0}^{N}$. 
We solve and simulate these equations numerically using {\tt fsolve} in MATLAB R2019a and present a few results. We verify numerically whether the solution obtained is indeed a saddle point using Proposition \ref{prop:suff-saddle-point}.
We will let $N=50, \step = 0.1, \Lambda = 0.1, \lambda = 1, \mu = 2$ for all simulations and  $\psi,\bar{\gr}_0, \bar{\vel}_0$ will be changed. 
These {\it initial conditions} define the optimal control problem. 
The disturbance and control are assumed to be unconstrained so as to keep the problem smooth, which makes the use of {\tt fsolve} easier.
We will use Remark \ref{rem:SOtwo} when plotting the optimal state trajectory. 
In the plots that we present, we employ a {\it counterclockwise} convention for $\theta$, which means, on going counterclockwise from $\theta = 0$, and transversing an angle $\pi/2$ we arrive at $\theta = \pi/2$, and on going clockwise and traversing an angle $\pi/2$ we arrive at $\theta = -\pi/2 \equiv \theta = 3\pi/2$. A positive value of $\vel$ therefore represents a velocity in the counterclockwise direction, and a negative value represents a velocity in the clockwise direction.

We present simulations of five trajectories the details for which are in Table \ref{table:simdata}. For an extensive simulation study refer \cite{anant-thesis}.
For \ref{list:S7minus} we show the optimal configuration, angular velocity and input sequences in Fig. \ref{fig:S7minus}. The initial condition starts at the opposite end at $\bar{\theta}_0 = \pi$ but manages to reach $\theta = 0$. This is an advantage of using a geometric controller which might or might not be guaranteed by conventional linear or nonlinear controllers.
Comparing between Fig. \ref{fig:S3-group} and \ref{fig:S16-group} we see that changing the direction of the angular velocity from counterclockwise in \ref{list:S3} to clockwise in \ref{list:S16} assists the controller in causing a drift towards $\theta = 0$. This is expected intuitively since the intended drift from $\bar{\theta}_0 = \pi/2$ to $\theta = 0$ is clockwise.
The numerical algorithm utilised by the solver requires an initial guess to begin the solution procedure. Between \ref{list:S17} and \ref{list:UW4}, we change the solver initial guess keeping everything else the same. The solver converges to a different saddle point in both cases. From Fig. \ref{fig:UW4-group} we observe that \ref{list:UW4} displays behaviour similar to quaternion unwinding (which is a phenomenon in which the system first diverges from the desired equilibrium then converges to it \cite{bhat-2000}) by first diverging away from $\theta = 0$ and then converging towards it. 

{\renewcommand\arraystretch{1.2}
\begin{table}[h!]
\centering
\begin{tabular}{ L P  K  K  K  M } 
 \toprule
 Trajectory  & $\psi$ & $\bar{\vel}_0$ & $\bar{\theta}_0$ & Figure & \\
   & & (rad/s) & (rad) &  &\\ \midrule
   \siml{list:S7minus} & 0.3 & 0.3 & 0.3 & Fig \ref{fig:S7minus} & \\ \midrule
   
   \siml{list:S3} & 0.2 & 0.1 & $\pi/2$ & Fig \ref{fig:S3-group} & \\
   
   \siml{list:S16} & 0.2 & -0.1 & $\pi/2$ & Fig \ref{fig:S16-group} & \\ \midrule
   
   \siml{list:S17} & 0.3 & -0.1 & $4\pi/3$ & Fig \ref{fig:S17-group} & \\
   
   \siml{list:UW4} & 0.3 & -0.1 & $4\pi/3$ & Fig \ref{fig:UW4-group} & \\
  \bottomrule
\end{tabular}
\caption{Simulation Data}
\label{table:simdata}
\end{table}
}


\begin{figure}[ht]
\begin{subfigure}{\displaytextwidth}
\centering
\includegraphics[scale = 0.5]{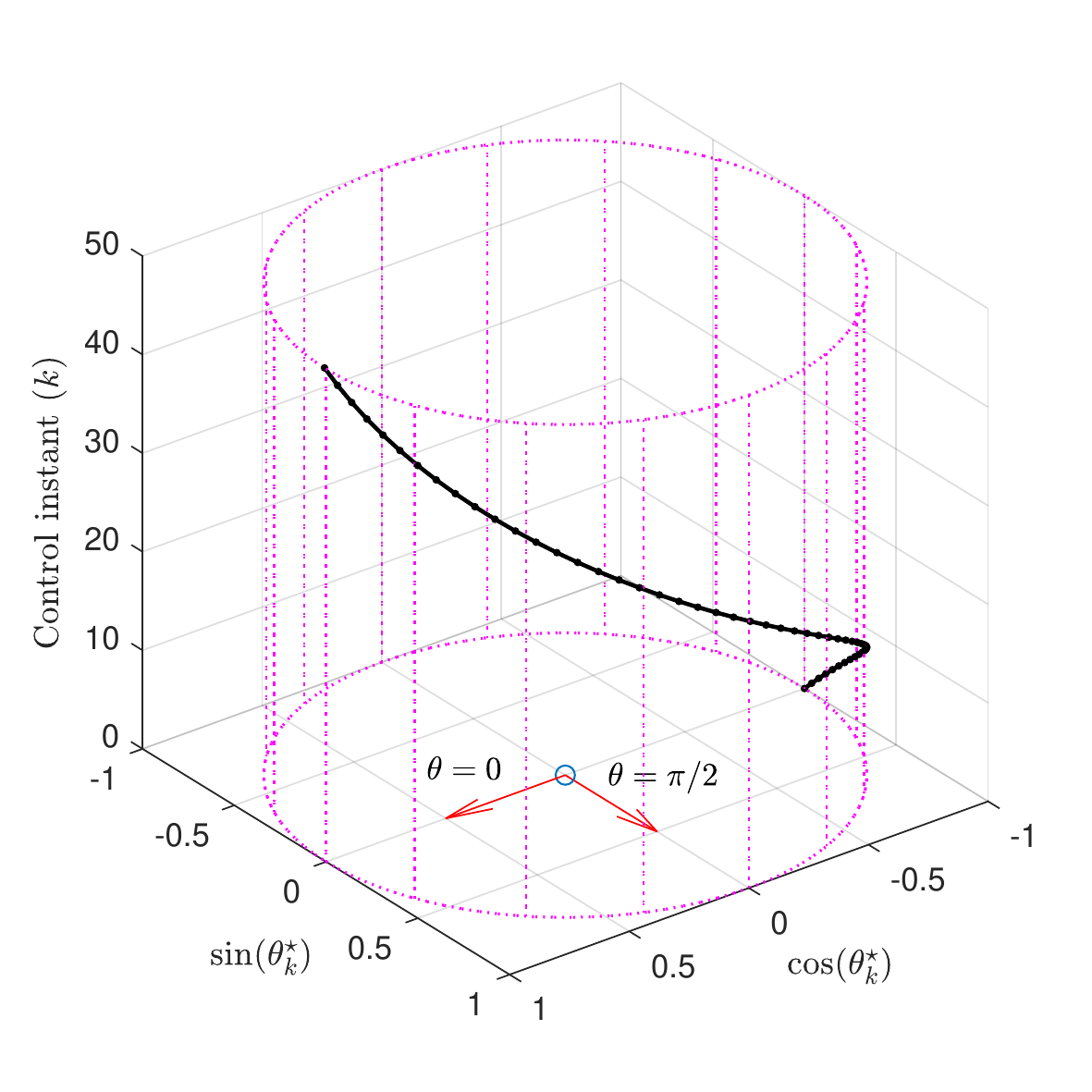}
\caption{Optimal configuration trajectory}
\end{subfigure}
\begin{subfigure}{\displaytextwidth}
\centering
\includegraphics[scale = 0.5]{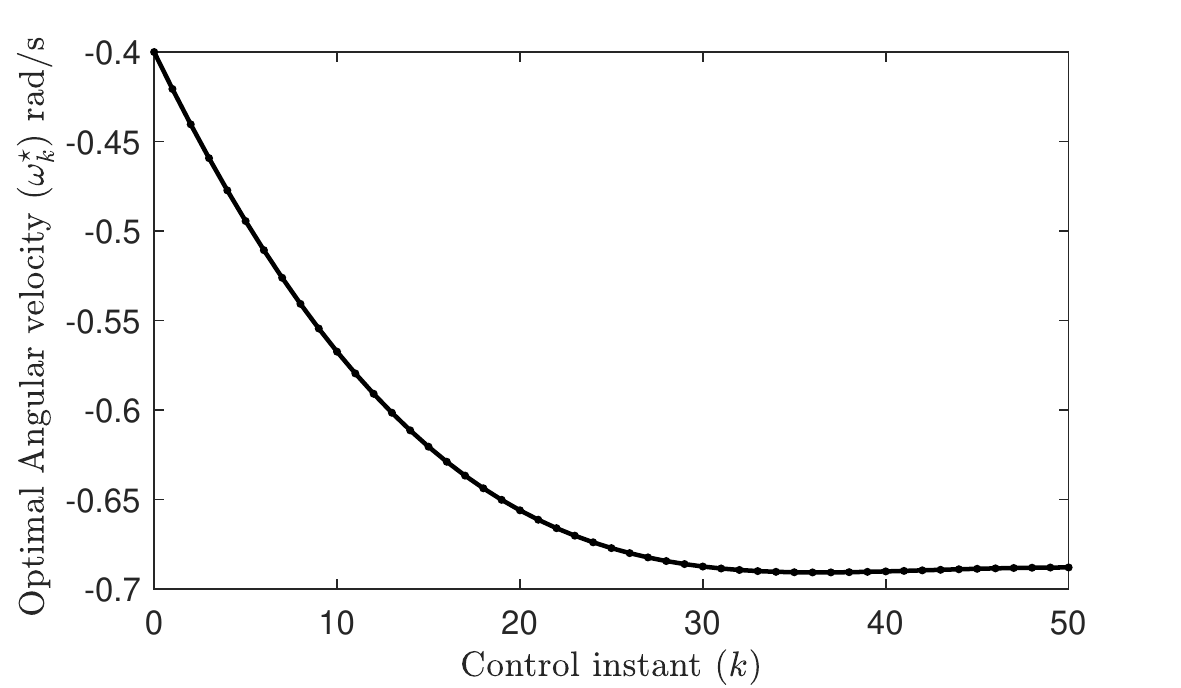}
\caption{Optimal angular velocity}
\end{subfigure}
\begin{subfigure}{\displaytextwidth}
\centering
\includegraphics[scale=0.5]{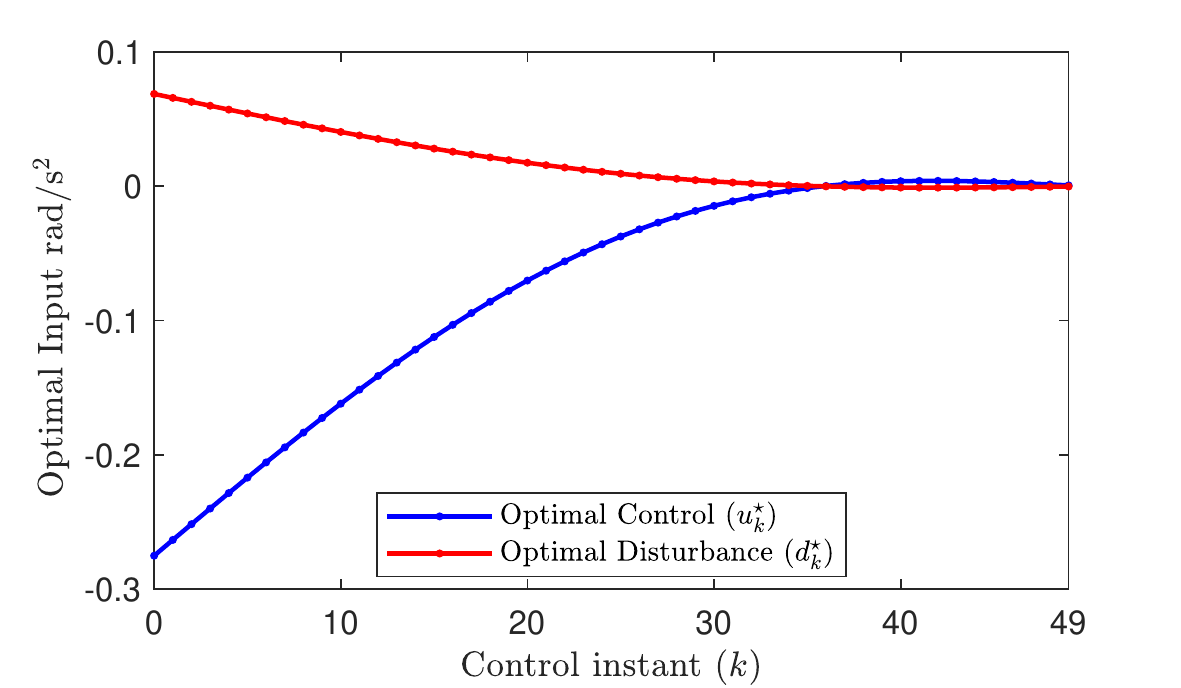}
\caption{Optimal input sequence}
\end{subfigure}
\caption{Trajectory \ref{list:S7minus}}
\label{fig:S7minus}
\end{figure}

\begin{figure}[ht]
\begin{subfigure}{\displaytextwidth}
\centering
\includegraphics[scale = 0.5]{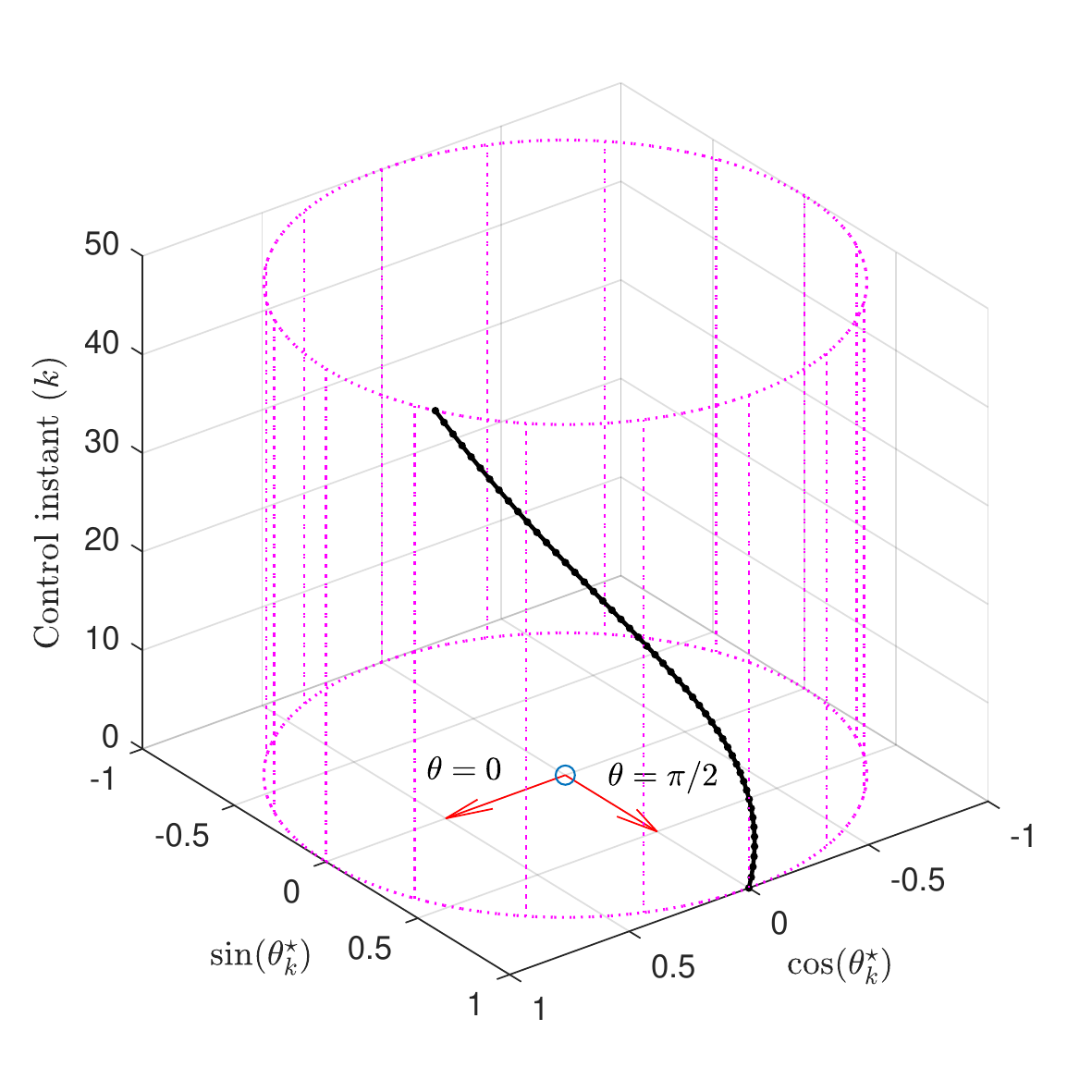}
\caption{Trajectory \ref{list:S3}}
\label{fig:S3-group}
\end{subfigure}
\begin{subfigure}{\displaytextwidth}
\centering
\includegraphics[scale = 0.5]{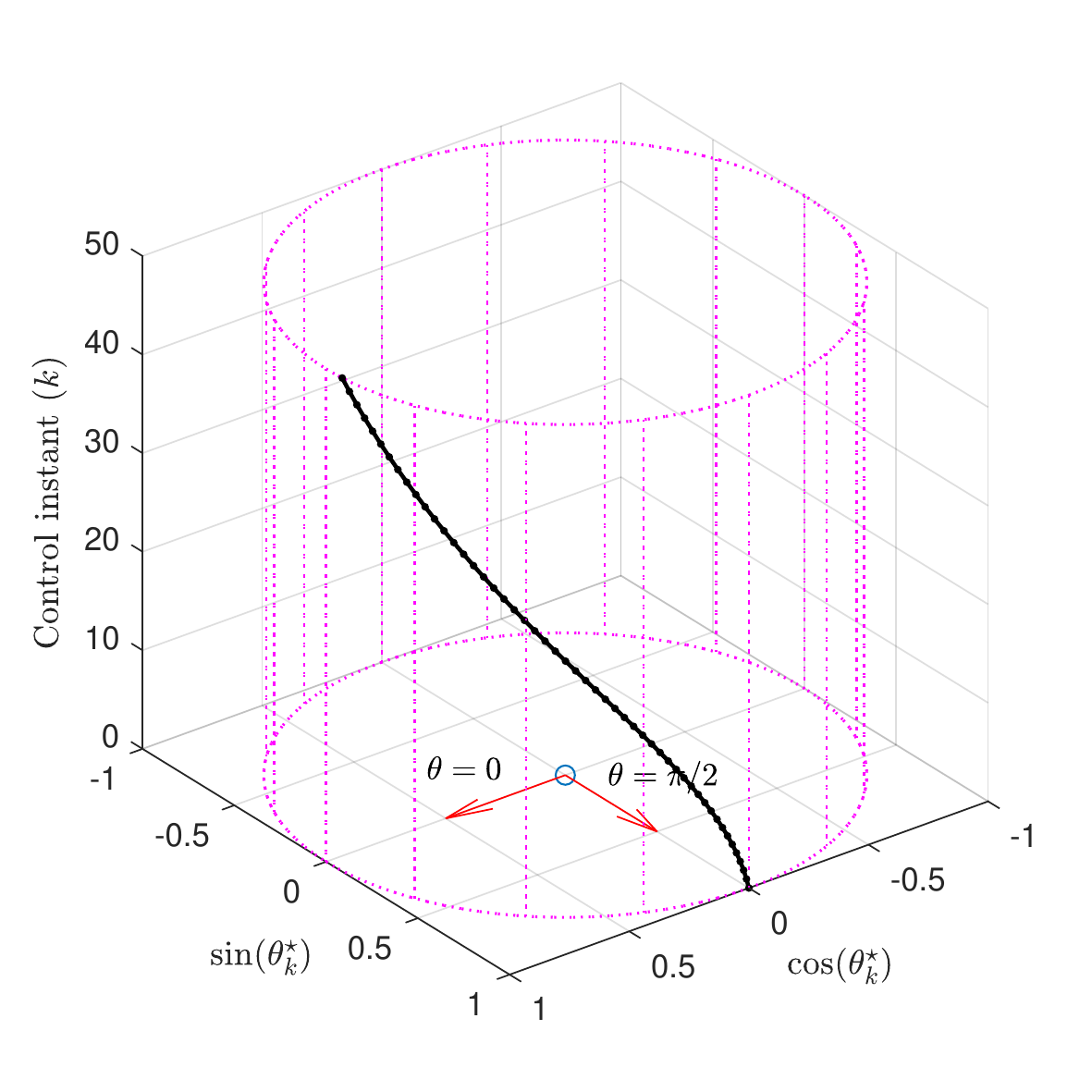}
\caption{Trajectory \ref{list:S16}}
\label{fig:S16-group}
\end{subfigure}
\caption{Optimal configuration trajectory}
\end{figure}

\begin{figure}[ht]
\begin{subfigure}{\displaytextwidth}
\centering
\includegraphics[scale = 0.5]{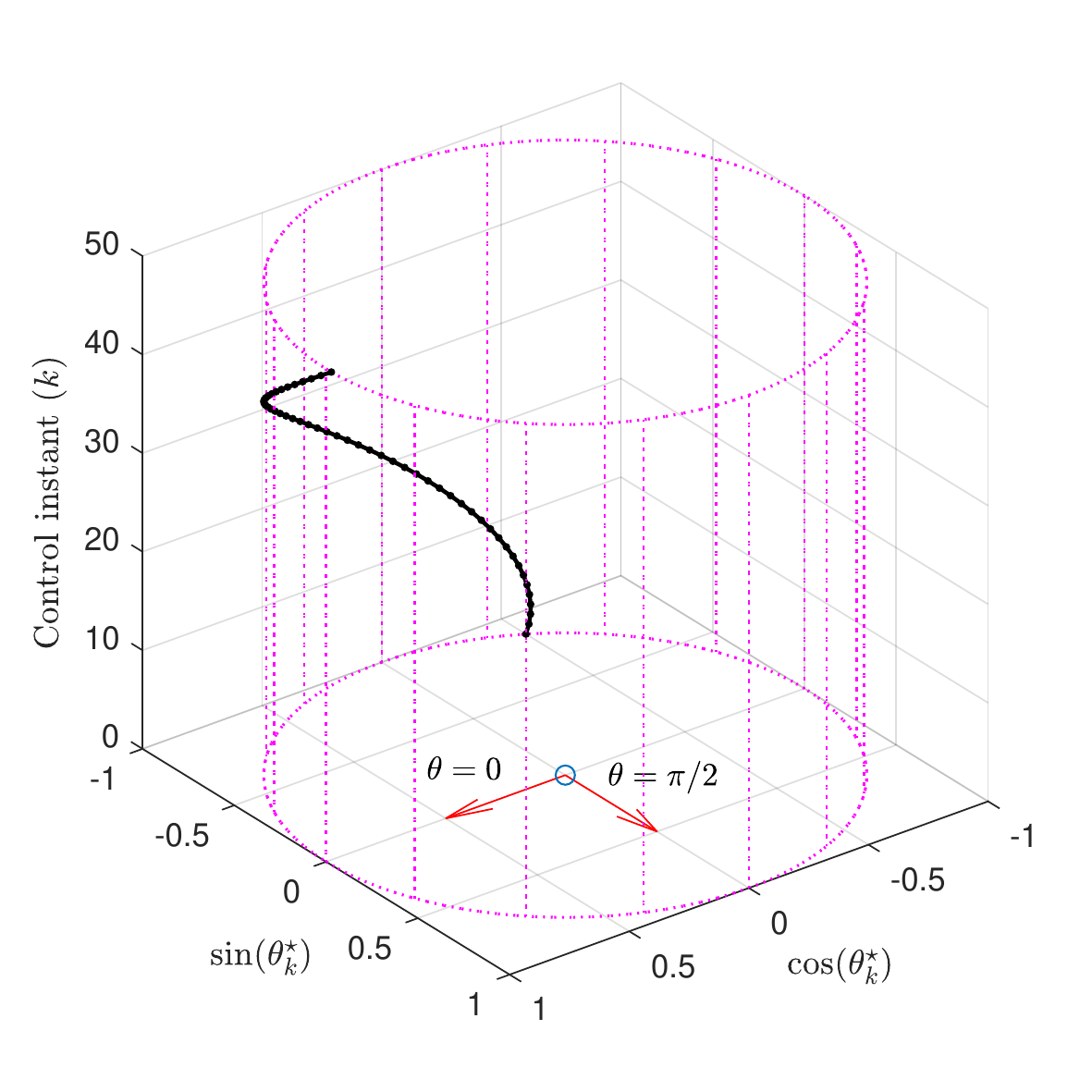}
\caption{Trajectory \ref{list:S17}}
\label{fig:S17-group}
\end{subfigure}
\begin{subfigure}{\displaytextwidth}
\centering
\includegraphics[scale = 0.5]{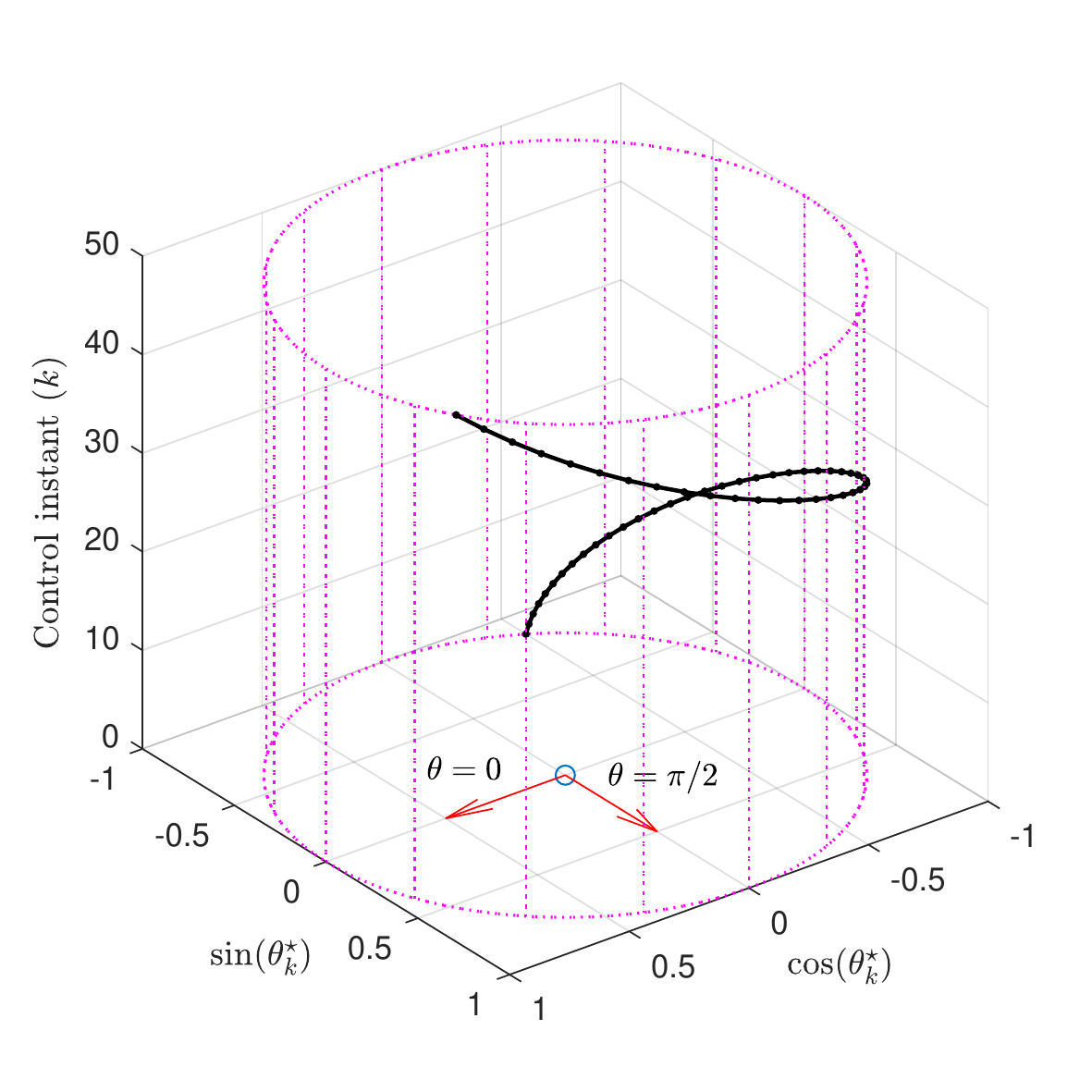}
\caption{Trajectory \ref{list:UW4}}
\label{fig:UW4-group}
\end{subfigure}
\caption{Optimal configuration trajectory}
\end{figure}


\appendix
\section{Proof of Proposition \ref{prop:saddle-point}}
\label{app:proof-saddle-prop}
We will first prove that Condition \ref{list:saddle-i} is equivalent to Condition \ref{list:saddle-ii}. 
In Condition \ref{list:saddle-ii}, the intersection is always non-empty since $(\opt{\con},\opt{\dist}) \in \Uc \times \D$. Assume that the intersection has one more point $(u_1,d_1) \ne (\opt{\con},\opt{\dist})$. Then necessarily $(\con_1,\dist_1) \in \Omega_1 \cup \Omega_2$. If $(\con_1,\dist_1) \in \Omega_1$ then $\dist_1 = \opt{\dist}$ and $\map (u_1,d_1) < \map(\opt{\con},\opt{\dist})$. Thus $(\opt{u},\opt{d})$ will not be a saddle point of $\map$ restricted to $\Uc \times \D$.  A similar conclusion can be drawn for the case when $(\con_1,\dist_1) \in\Omega_2$. Hence Condition \ref{list:saddle-i} $\implies$ \ref{list:saddle-ii}.

Assume that $(\opt{u},\opt{d})$ is not a saddle point. Then there exists $(\opt{u},\opt{d}) \ne (\con_1,\dist_1) \in \Uc \times \D$ such that either $\dist_1 = \opt{\dist}$ and $\map(\con_1,\dist_1) < \map(\opt{u},\opt{d})$ or $\con_1 = \opt{\con}$ and $\map(\opt{u},\opt{d}) < \map(\con_1,\dist_1)$. Then $(\con_1,\dist_1) \in \Omega_1 \cup \Omega_2$. Then $\big(\Omega_1 \cup \Omega_2 \cup \{ (\opt{\con},\opt{\dist}) \} \big) \cap (\Uc \times \D ) = \{ (\opt{\con},\opt{\dist}), (\con_1,\dist_1) \}$. Hence Condition \ref{list:saddle-ii} $\implies$ \ref{list:saddle-i}.

We will now prove that Condition \ref{list:saddle-ii} is equivalent to Condition \ref{list:saddle-iii}. Observe that 
\begin{align*}
\Omega_1^\prime \cup \Omega_2^\prime &= \{ (\opt{\con},\opt{\dist}) \} \cup \Omega_1 \cup \Omega_2 \cup \{ (\opt{\con},\opt{\dist}) \},  \\
&=\Omega_1 \cup \Omega_2 \cup \{ (\opt{\con},\opt{\dist}) \} 
\end{align*}
Therefore,
\begin{align*}
\big(\Omega_1 \cup \Omega_2 \cup \{ (\opt{\con},\opt{\dist}) \} \big) \cap (\Uc \times \D ) = \{ (\opt{\con},\opt{\dist}) \},& \\
\iff (\Omega_1^\prime \cup \Omega_2^\prime) \cap (\Uc \times \D ) = \{ (\opt{\con},\opt{\dist}) \},& \\
\iff \big(\Omega_1^\prime \cap (\Uc \times \D \big) \cup \big(\Omega_2^\prime \cap (\Uc \times \D \big) = \{ (\opt{\con},\opt{\dist}) \}, & \\
\iff 
\begin{cases}
\big(\Omega_1^\prime \cap (\Uc \times \D \big)  &= \{ (\opt{\con},\opt{\dist}) \} \\
 \big(\Omega_2^\prime \cap (\Uc \times \D \big)  &= \{ (\opt{\con},\opt{\dist}) \}
\end{cases} &
\end{align*}
The last double implication is justified as follows. We have two sets, both non-empty (since they have at least one common point which is $(\opt{u},\opt{d})$. Their intersection is a singleton set. Hence both of them must individually be that singleton set.

\section{}
\label{app:scaling}

In this section we show that given $r \in \R_{>0}$, if the Hamiltonian in \eqref{eq:ham-con} is defined using $r{\check{\nu}}$ as
\begin{align}
&\upto{N-1} \times \LA^* \times (\R^{\dimg})^* \times \LG \times \R^{\dimg} \times \R^{\dimu}  \ni (k,\check{\zeta},\check{\xi},\gr,\vel,\con) \notag \\  
&\mapsto \hamu^{r\check{\nu}}(k,\check{\zeta},\check{\xi},\gr,\vel,\con) \defn r\check{\nu} c_k(\gr,\vel,\con,\opt{\dist}_k) \notag \\ 
 & \qquad \quad + \inprod{\check{\zeta}}{\exp^{-1}(\kin(\gr,\vel))}  + \inprod{\check{\xi}}{\dyn(\gr,\vel,\con,\opt{\dist}_k)},
\end{align}
and the covectors are changed to $(r\check{\zeta})_{k=0}^{N-1}$ and $(r\check{\xi})_{k=0}^{N-1}$, then  conditions \ref{list:m-i}-\ref{list:m-v} are still satisfied.

To show this, first define $\opt{\check{\Gamma}}_k \defn (r\check{\zeta}_k,r\check{\xi}_k,\opt{\gr}_k,\opt{\vel}_k,\opt{\con}_k)$.
Condition \ref{list:m-i} is satisfied since 
\begin{align*}
\deriv{\check{\zeta}}{\hamu^{\check{\nu}}}{\opt{\check{\gamma}}_k} = \deriv{\check{\zeta}}{\hamu^{r\check{\nu}}}{\opt{\check{\Gamma}}_k}, \quad \deriv{\check{\xi}}{\hamu^{\check{\nu}}}{\opt{\check{\gamma}}_k} = \deriv{\check{\xi}}{\hamu^{r\check{\nu}}}{\opt{\check{\Gamma}}_k} .
\end{align*}
Conditions \ref{list:m-ii} and \ref{list:m-iv} are satisfied since $\Ad_{\gr}^*(\cdot)$ is a linear map for any $\gr \in \LG$, Remark \ref{rem:bij-lin} holds and since
\begin{gather*}
r\deriv{\con}{\hamu^{\check{\nu}}}{\opt{\check{\gamma}}_{k-1}} = \deriv{\con}{\hamu^{r\check{\nu}}}{\opt{\check{\Gamma}}_{k-1}}, \\ 
r\deriv{\vel}{\hamu^{\check{\nu}}}{\opt{\check{\gamma}}_k} = \deriv{\vel}{\hamu^{r\check{\nu}}}{\opt{\check{\Gamma}}_k}, \quad 
r\tang{\opt{\check{\gamma}}_k }{\gr}{\hamu^{\check{\nu}}} = \tang{\opt{\check{\Gamma}}_k }{\gr}{\hamu^{r\check{\nu}}}.
\end{gather*}
Condition \ref{list:m-iv} is satisfied because $\check{\zeta}^{N-1}$ and $\check{\xi}^{N-1}$ depend linearly on $\check{\nu}$.
Condition \ref{list:m-v} is trivially satisfied.
\section{}

\label{app:ex-two-adj}

We will present the derivation of \eqref{eq:ocp2-zeta-dyn}. 
As noted in Remark \ref{rem:SOtwo}, let $\opt{\gr}_k = \exp(\opt{\theta}_k)$ for some $\opt{\theta}_k \in [0,2\pi)$. For arbitrary $t \in \R$, define $v \coloneqq \sigma(t) \in \sotwo$. 
Recall that 
\begin{equation*}
\exp(v) = \begin{pmatrix}\cos(t) & -\sin(t) \\ \sin(t) & \cos(t) \end{pmatrix},
\end{equation*}
and that for any symmetric $S \in \R^{2 \times 2}$, $\tr(Sv) = \frac{1}{2}\tr((S - S^{\top})v)$. 
Then,
\begin{align*}
&\inprod{\tangostar{\id}{\Phi_{\opt{\gr}_k}}(\tang{\opt{\gamma}_k}{\gr}{\ham} )}{v} = \inprod{\tang{\opt{\gamma}_k}{\gr}{\ham} }{\tango{\id}{\Phi_{\opt{\gr}_k}} v}, \\
& = \left. \frac{d}{ds} \right\rvert_{s=0} \ham(\opt{\gr}_k \exp(vs)) = \frac{1}{2} \left. \frac{d}{ds} \right\rvert_{s=0} \psi^2\tr\left(\opt{\gr}_k \exp(vs)\right), \\
&= \left. \frac{d}{ds} \right\rvert_{s=0} \psi^2 \cos(\theta + ts) 
= -t \psi^2 \sin(\theta), \\
&= \frac{1}{2}\tr(\psi^2 \opt{\gr}_kv) = \frac{1}{2}\tr(\left(\psi^2 \frac{(\opt{\gr}_k)^{\top} - \opt{\gr}_k}{2}\right)^{\top}v), \\
&= \inprod{\left(\psi^2 \frac{(\opt{\gr}_k)^{\top} - \opt{\gr}_k}{2}\right)}{v}.
\end{align*}
\eqref{eq:ocp2-zeta-trans} is derived similarly.
\section{}
\label{app:lq-dyn-game}
If $\psi = 0$ the adjoint equations simplify to 
\begin{align*}
&\zeta^{k} = 0, \: 
&\xi^{k} = -\sum_{i=k+1}^{N} \Lambda^2\opt{\vel}_i \quad \forall k \in \upto{N-1}.
\end{align*}
The optimal inputs retain the same expressions as \eqref{eq:ocp2-control}, \eqref{eq:ocp2-disturbance}. We continue to assume that $\mu,\lambda,\Lambda > 0$ are carefully chosen so as to ensure that the cost function in \eqref{eq:opt-prob-ex-cost} admits a saddle point \cite[Lemma 3.1]{basar-1995}. As per \cite[Theorem 3.1]{basar-1995}, the solutions for the optimal inputs, for $k \in \upto{N-1}$ are (to maintain consistency between results, we will, without any loss of generality, let $\lambda = 1$)
\begin{align*}
\opt{\con}_k &= -\step M_{k+1}L_k^{-1}\vel_k, \\
\opt{\dist}_k &= \mu^{-2}\step M_{k+1}L_k^{-1}\vel_k, \\
\opt{\vel}_{k+1} &= L_k^{-1}\opt{\vel}_k, \quad \opt{\vel}_{0} = \bar{\vel}_0, \\
M_{k} &= \Lambda^2 + M_{k+1}L_k^{-1}, \quad M_{N} = \Lambda^2, \\
L_k &\coloneqq 1 + \step^2 M_{k+1}(1 - \mu^{-2}) .
\end{align*}
To show that this solution satisfies our necessary conditions \eqref{eq:ocp2-control} and \eqref{eq:ocp2-disturbance}, we will utilise mathematical induction.

{\it Base Case:} 
\begin{align*}
\opt{\con}_{N-1} &= -\step \Lambda^2 L_{N-1}^{-1}\opt{\vel}_{N-1}, \\
&= -\step \Lambda^2 \opt{\vel}_{N}, \\
&= \step \xi^{N-1}.
\end{align*}

{\it Induction Hypothesis:} Assume the claim to be true for $k = N-1,N-2,\ldots,i$.

{\it Induction Step:}
\begin{align*}
\opt{\con}_{i-1} &= -\step M_{i}L_{i-1}^{-1}\opt{\vel}_{i-1} \\
&= -\step (\Lambda^2 + M_{i+1}L_i^{-1}) L_{i-1}^{-1}\opt{\vel}_{i-1}, \\
&= -\step \Lambda^2 \opt{\vel}_{i} -\step M_{i+1}L_{i}^{-1}\opt{\vel}_{i}, \\
&= -\step \Lambda^2 \opt{\vel}_{i} + \opt{\con}_{i} = -\step \Lambda^2 \opt{\vel}_{i}  -\Lambda^2\sum_{k = i+1}^{N}\opt{\vel}_k, \\
&= \step \xi^{i-1}.
\end{align*}
This verifies our claim.

\bibliographystyle{siam}
\bibliography{references}

\begin{thebibliography}{10}

\bibitem{agrachev}
{\sc A.~A. Agrachev and Y.~L. Sachkov}, {\em Control Theory from the Geometric
  Viewpoint}, vol.~87 of Encyclopaedia of Mathematical Sciences,
  Springer-Verlag, Berlin, 2004.
\newblock Control Theory and Optimization, II.

\bibitem{basar-1995}
{\sc T.~Basar and P.~Bernhard}, {\em {$H^{\infty}$}-Optimal Control and Related
  Minimax Design Problems}, Systems \& Control: Foundations \& Applications,
  Birkh\"{a}user Boston, Inc., Boston, MA, second~ed., 1995.
\newblock A Dynamic Game Approach.

\bibitem{basar-1998}
{\sc T.~Basar and G.~J. Olsder}, {\em Dynamic {N}oncooperative {G}ame {T}heory,
  2nd Ed.}, Society for Industrial and Applied Mathematics, 1998.

\bibitem{bernhard-2015}
{\sc P.~Bernhard}, {\em Pursuit-evasion games and zero-sum two-person
  differential games}, in Encyclopedia of Systems and Control, 2015.

\bibitem{bhat-2000}
{\sc S.~P. Bhat and D.~S. Bernstein}, {\em A topological obstruction to
  continuous global stabilization of rotational motion and the unwinding
  phenomenon}, Systems \& Control Letters, 39 (2000), pp.~63 -- 70.

\bibitem{bloch-2015}
{\sc A.~Bloch, L.~Colombo, R.~Gupta, and D.~M. de~Diego}, {\em A Geometric
  Approach to the Optimal Control of Nonholonomic Mechanical Systems}, Springer
  International Publishing, 2015, pp.~35--64.

\bibitem{bolt-robust}
{\sc V.~G. Boltyanski and A.~S. Poznyak}, {\em The Robust Maximum Principle},
  Systems \& Control: Foundations \& Applications, Birkh\"{a}user/Springer, New
  York, 2012.
\newblock Theory and applications.

\bibitem{bolt}
{\sc V.~Boltyanskii}, {\em Optimal Control of Discrete Systems}, Wiley, New
  York, 1978.

\bibitem{bressan-2011}
{\sc A.~Bressan}, {\em Noncooperative differential games}, Milan Journal of
  Mathematics, 79 (2011), pp.~357--427.

\bibitem{chen-1997}
{\sc H.~{Chen}, C.~W. {Scherer}, and F.~{Allgower}}, {\em A game theoretic
  approach to nonlinear robust receding horizon control of constrained
  systems}, in Proceedings of the 1997 American Control Conference (Cat.
  No.97CH36041), vol.~5, June 1997, pp.~3073--3077 vol.5.

\bibitem{chern-2005}
{\sc F.~L. Chernousko}, {\em Minimax control for a class of linear systems
  subject to disturbances1}, Journal of Optimization Theory and Applications,
  127 (2005), pp.~535--548.

\bibitem{galperin-2008}
{\sc E.~Galperin}, {\em The isaacs equation for differential games, totally
  optimal fields of trajectories and related problems}, Computers \&
  Mathematics with Applications, 55 (2008), pp.~1333 -- 1362.

\bibitem{gueler}
{\sc O.~G\"{u}ler}, {\em Foundations of Optimization}, Springer-Verlag New
  York, 2010.

\bibitem{halmos}
{\sc P.~Halmos}, {\em Finite-Dimensional Vector Spaces}, Springer-Verlag New
  York, 1st~ed., 1958.

\bibitem{holm}
{\sc D.~D. Holm, T.~Schmah, and C.~Stoica}, {\em Geometric Mechanics and
  Symmetry}, vol.~12 of Oxford Texts in Applied and Engineering Mathematics,
  Oxford University Press, Oxford, 2009.

\bibitem{bloch-2006}
{\sc I.~I. {Hussein}, M.~{Leok}, A.~K. {Sanyal}, and A.~M. {Bloch}}, {\em A
  discrete variational integrator for optimal control problems on so(3)}, in
  Proceedings of the 45th IEEE Conference on Decision and Control, Dec 2006,
  pp.~6636--6641.

\bibitem{fernando-2013}
{\sc F.~Jim{\'e}nez, M.~Kobilarov, and D.~Mart{\'i}n~de Diego}, {\em Discrete
  variational optimal control}, Journal of Nonlinear Science, 23 (2013),
  pp.~393--426.

\bibitem{anant-thesis}
{\sc A.~A. Joshi}, {\em Discrete-time robust optimal control on matrix lie
  groups: A pontryagin maximum principle approach}, Master's thesis, Indian
  Institute of Technology, Bombay, 2020.

\bibitem{gupta-2019}
{\sc R.~Kipka and R.~Gupta}, {\em The discrete-time geometric maximum
  principle}, SIAM Journal on Control and Optimization, 57 (2019),
  pp.~2939--2961.

\bibitem{lee}
{\sc J.~M. Lee}, {\em Introduction to Smooth Manifolds}, Springer-Verlag New
  York, 2nd~ed., 2003.

\bibitem{liberzon}
{\sc D.~Liberzon}, {\em Calculus of Variations and Optimal Control Theory},
  Princeton University Press, Princeton, NJ, 2012.
\newblock A concise introduction.

\bibitem{marsden}
{\sc J.~Marsden and T.~Ratiu}, {\em Introduction to Mechanics and Symmetry},
  Springer-Verlag New York, 2nd~ed., 1999.

\bibitem{marsden-2001}
{\sc J.~E. Marsden and M.~West}, {\em Discrete mechanics and variational
  integrators}, Acta Numerica, 10 (2001), p.~357–514.

\bibitem{mishal-2020}
{\sc {Mishal Assif P. K.}, D.~Chatterjee, and R.~Banavar}, {\em A simple proof
  of the discrete time geometric {P}ontryagin maximum principle on smooth
  manifolds}, Automatica, 114 (2020), p.~108791.

\bibitem{prad-2019}
{\sc P.~{Paruchuri} and D.~{Chatterjee}}, {\em Discrete time {P}ontryagin
  maximum principle under state-action-frequency constraints}, IEEE
  Transactions on Automatic Control, 64 (2019), pp.~4202--4208.

\bibitem{karmvir-2018-jgcd}
{\sc K.~S. Phogat, D.~Chatterjee, and R.~Banavar}, {\em Discrete-time optimal
  attitude control of a spacecraft with momentum and control constraints},
  Journal of Guidance, Control, and Dynamics, 41 (2018), pp.~199--211.

\bibitem{karmvir-2018}
{\sc K.~S. Phogat, D.~Chatterjee, and R.~N. Banavar}, {\em A discrete-time
  {P}ontryagin maximum principle on matrix {L}ie groups}, Automatica, 97
  (2018), pp.~376 -- 391.

\bibitem{pontryagin}
{\sc L.~S. Pontryagin, V.~G. Boltyanskii, R.~V. Gamkrelidze, and E.~F.
  Mishchenko}, {\em Mathematical Theory of Optimal Processes}, CRC Press, 1987.

\bibitem{raimondo-2009}
{\sc D.~M. Raimondo, D.~Limon, M.~Lazar, L.~Magni, and E.~F. ndez Camacho},
  {\em Min-max model predictive control of nonlinear systems: A unifying
  overview on stability}, European Journal of Control, 15 (2009), pp.~5 -- 21.

\bibitem{sarg-2000}
{\sc R.~W.~H. Sargent}, {\em Optimal control}, vol.~124, 2000, pp.~361--371.

\bibitem{spivak}
{\sc M.~Spivak}, {\em Calculus on Manifolds}, W. A. Benjamin, Inc., New
  York-Amsterdam, 1965.

\bibitem{vinter-2005}
{\sc R.~B. Vinter}, {\em Minimax optimal control}, SIAM Journal on Control and
  Optimization, 44 (2005), pp.~939--968.

\end{thebibliography}

\end{document}